\documentclass[11pt,reqno]{amsart}
\usepackage{amsmath,amssymb,amsthm}
\usepackage{paralist,subfigure}
\usepackage[misc]{ifsym}
\usepackage{epsfig} 
\usepackage{epstopdf} 
\usepackage[colorlinks=true]{hyperref}
\hypersetup{urlcolor=blue, citecolor=red}
\usepackage[pagewise]{lineno}
\allowdisplaybreaks

\newtheorem{theo}{Theorem}[section]

\newtheorem{lemm}[theo]{Lemma}
\newtheorem{prop}[theo]{Proposition}
\newtheorem{assumption}[theo]{Assumption}
\newtheorem{rema}[theo]{Remark}

\numberwithin{equation}{section}
\newcommand{\dd}{{\rm d}}

\title[Superiority of symplectic methods]
{Superiority of symplectic methods for stochastic Hamiltonian system via asymptotic error distribution} 

\author[Jialin Hong and Ge Liang and Derui Sheng]{}


\thanks{The first author is supported by National key R\&D Program of China (No.2020YFA0713701), and by the National Natural Science Foundation of China (Nos. 11971470, 12031020, 12171047).}

\thanks{}

\begin{document}
\maketitle

\centerline{\scshape
Jialin Hong$^{{\href{hjl@lsec.cc.ac.cn}{\textrm{\Letter}}}1,2}$, Ge Liang$^{{\href{liangge2020@lsec.cc.ac.cn}{\textrm{\Letter}}}1,2}$ and Derui Sheng$^{{\href{derui.sheng@polyu.edu.hk}{\textrm{\Letter}}}*3}$}

\medskip

{\footnotesize
 \centerline{$^1$Academy of Mathematics and Systems Science, Chinese Academy of Sciences, Beijing
	100190, China}
} 

\medskip

{\footnotesize
 \centerline{$^2$School of Mathematical Sciences, University of Chinese Academy of
	Sciences, Beijing 100049, China}
}
\medskip
{\footnotesize
 \centerline{$^3$Department of Applied Mathematics, The Hong Kong Polytechnic University, Hung Hom, Kowloon, Hong Kong}
}

\bigskip

 \centerline{(Communicated by Handling Editor)}


\begin{abstract}
The superiority of symplectic methods for stochastic Hamiltonian systems has been widely recognized, yet the probabilistic mechanism behind this superiority remains incompletely understood.
This paper studies the superiority of symplectic methods from the perspective of the asymptotic error distribution, i.e., the limit distribution of normalized error. 
Focusing on stochastic Hamiltonian systems driven by additive noise, we obtain the asymptotic limit of the normalized error distribution of the $\theta$ method $(\theta\in[0,1])$ that is symplectic if and only if $\theta=\frac12$.
By establishing upper bounds for the second-order moment of the asymptotic error distribution, we show that the midpoint method minimizes the error constant of the $\theta$ method for a large time horizon $T$. Furthermore, we take the linear stochastic oscillator as a test equation and investigate exact asymptotic error constants of several symplectic and non-symplectic methods. Our result suggests that in the long-time computation, the probability that the error deviates from zero decays exponentially faster for the symplectic methods
than that for the non-symplectic ones.

\end{abstract}



\section{Introduction}
Consider the following $2d$-dimensional stochastic Hamiltonian system:
\begin{equation}\label{SHE-intro}
	\dd \begin{pmatrix}
		X_t^1\\
		X_t^2
	\end{pmatrix}=\begin{pmatrix}
	0&I_d\\
	-I_d&0	
\end{pmatrix}\nabla H(X_t^1,X_t^2)\dd t+\sum_{k=1}^m\begin{pmatrix}
	0&I_d\\
	-I_d&0	
\end{pmatrix}\nabla H_k(X_t^1, X_t^2) \dd W_t^k
\end{equation}
for $t\in [0,T]$
with the initial data $(X_0^1,X_0^2)\in\mathbb R^d\times \mathbb R^d$. Here, $H, H_1,\ldots, H_m:\mathbb{R}^{2d}\rightarrow\mathbb{R}$ are the Hamiltonians and $W=(W^1,\ldots,W^m)^\top$ is an $m$-dimensional Brownian motion defined a complete filtered probability space $\big(\Omega,\mathcal F,\{\mathcal F_t\}_{t\in[0,T]},\mathbb P\big)$ with the filtration $\{\mathcal F_t\}_{t\in[0,T]}$ satisfying the usual conditions. One of the most intrinsic properties for \eqref{SHE-intro} is that its phase flow preserves the symplectic structure in phase space, i.e., $\dd X_t^1\wedge\dd X_t^2=\dd X_0^1\wedge\dd X_0^2$, $t\in[0,T]$ for almost surely $\omega\in\Omega$ (see, e.g., \cite{Bismut}). Such a property is called symplecticity, which implies that the sum of the oriented areas of the projections of phase flow onto each coordinate planes $(x_i,y_i)$, $i=1,2,\ldots,d$, is invariant.

Stochastic Hamiltonian systems have extensive applications in various fields, such as chemistry, physics, and celestial mechanics. A basic principle in designing efficient numerical methods for \eqref{SHE-intro} is that the numerical method should preserve the symplecticity of the phase flow of \eqref{SHE-intro}. Such a numerical method is called a symplectic method, originating from the pioneering work of Milstein et al. (see, e.g., \cite{Milstein2004}).
Extensive numerical simulations (see, e.g., \cite{cui2017stochastic,hong2023symplectic,hong2019invariant,wang2016modified}) show that when applied to stochastic Hamiltonian systems, symplectic methods exhibit long-time stability compared to non-symplectic methods. The underlying mechanism behind the superiority of symplectic methods for stochastic Hamiltonian systems has attracted considerable attention recently.
Inspired by deterministic systems, some studies have explained the long-term stability of symplectic methods through modified equations and backward error analysis techniques (see, e.g., \cite{wang2016modified,computingeffective2018}). From a probabilistic standpoint, \cite{suijizhenzidapiancha,LDPxde} investigated this issue by proving that symplectic methods can asymptotically preserve the large deviation principles of key physical observables associated with stochastic Hamiltonian systems, while \cite{chen2024superioritystochasticsymplecticmethods} addresses this issue from the perspective of the law of iterated logarithm. 
Following this research line, we study the asymptotic error distributions of numerical methods to reveal the superiority of symplectic methods for stochastic Hamiltonian systems.

The asymptotic error distribution characterizes quantitively the probabilistic evolution of the error process $\{U^N_t\}_{t\in[0,T]}$ between the numerical and exact solutions
as the step-size tends to zero. 
Extensive results have been established for various stochastic systems. 
For instance, \cite{JP98} proved that for stochastic differential equations with Lipschitz nonlinearity and multiplicative noise, the normalized error process $\{\sqrt{N}U^N_t\}_{t\in[0,T]}$ of the Euler–Maruyama method converges in distribution to some process $U=\{U_t\}_{t\in[0,T]}$. This result was later extended in \cite{PQM2020} to equations with locally Lipschitz nonlinearities.
The exact rate of convergence of numerical methods for differential equations driven by fractional Brownian motions was investigated in, e.g., \cite{HLN2016, NN2007, ZHL2023,UK25}. For more related works, we also refer to \cite{FU2023, NS2023} for the Euler method of stochastic Volterra equations and to \cite{hong2024asymptoticerrordistributionaccelerated} for the accelerated exponential Euler method of stochastic partial differential equations. 
Beyond numerical accuracy, asymptotic error distributions provide deeper insights into the error structure of numerical methods \cite{BN2006}. Prior work established that the limiting error process for the Euler–Maruyama method forms a gradient in the Dirichlet form sense, enabling error analysis via local Dirichlet forms \cite{BN08,BN2021}. Such structures play a crucial role in error propagation in Monte Carlo simulations, particularly in financial modeling (see, e.g., \cite{BN08}).

In this work, we focus on the stochastic Hamiltonian system with additive noise (i.e., \eqref{SHE-intro} with affine $\{H_k\}_{k=1}^m$) and study the asymptotic error distribution of the $\theta$ method $(\theta\in[0,1])$.
The $\theta$ method is symplectic for \eqref{SHE-intro} if and only if $\theta=\frac12$, corresponding to the midpoint method.
Since this method exhibits first-order strong convergence for the additive noise case, the normalized error is defined using a normalization constant $N$, rather than $\sqrt N$ that is commonly used for multiplicative noise (see, e.g., \cite{JP98, PQM2020}). 
The normalized error process can be decomposed into a negligible part that vanishes identically in probability and a dominant part that converges in distribution to the solution of a stochastic differential equation. This suggests that
 the sharp strong convergence order of the $\theta$ method is $1$ for \eqref{SHE-intro} with additive noise, regardless of the value of $\theta\in[0,1]$. We further provide in Theorem \ref{daa} an upper bound for the second-order moment of the asymptotic error distribution, which depends on $\theta$ and $T$. This bound is minimized when $\theta=\frac12$ for large $T$, which implies that the midpoint method has the smallest asymptotic error constant among all 
$\theta$-methods.

Inspired by \cite{suijizhenzidapiancha}, we take the linear stochastic oscillator as a test equation to further investigate the exact asymptotic error constants of several numerical methods. In detail, we derive the error constant $K_T$ for several concrete symplectic and non-symplectic methods for the linear stochastic oscillator, and find that the growth of $K_T$ is almost proportional to $T$ and $T^3$ for the considered symplectic and non-symplectic methods, respectively. 
Consequently, at the scale $\epsilon$, the probability of the error's deviation from the zero decays exponentially faster for the symplectic methods
than that of the non-symplectic methods.
This comparison reveals the superiority of symplectic methods over non-symplectic methods in the long-time computation from the perspective of the asymptotic error distribution.
Based on these findings, we plan to extend our investigation to error structures of symplectic methods for stochastic Hamiltonian systems in future work.

The rest of this paper is organized as follows. In section \ref{sec:2}, we establish the asymptotic error distribution of the $\theta$ method for \eqref{SHE-intro} with additive noise. By taking the linear stochastic oscillator as a test equation, we further study the asymptotic error distributions of several symplectic and non-symplectic methods in sections \ref{chap3}
 and \ref{chap4}. Numerical experiments are finally performed in section \ref{chap5} to verify the theoretical results.

\section{Asymptotic error distribution of $\theta$ method}\label{sec:2}
In this section, we investigate the asymptotic error distribution of the normalized error for the $\theta$ method applied to the stochastic Hamiltonian system \eqref{SHE-intro} with additive noise. Specifically, we consider the following model
\begin{equation}\label{SHE}
	\dd X_t=b(X_t)\dd t+\sigma \dd W_t,\quad t\in [0,T],
\end{equation}
where $\sigma\in\mathbb{R}^{2d\times m}$ is a constant matrix and 
$$b:=J\nabla H\quad \text{with} \quad J:=\begin{pmatrix}
	0&I_d\\
	-I_d&0	
\end{pmatrix}.$$

%

\begin{assumption}\label{ass1}
	The Hamiltonian $H$ has continuous bounded derivatives up to order $4$.
\end{assumption}
By introducing a uniform partition of $[0,T]$ with step-size $h=\frac{T}{N}$, where $N\in\mathbb{N}_+$, the $\theta$ method applied to \eqref{SHE} reads
\begin{equation}\label{stochastic theta scheme}
	\left\{\begin{split}
		\widehat{X}^{N,\theta}_{k+1}&=\widehat{X}^{N,\theta}_k+h b\Big(\theta \widehat{X}^{N,\theta}_{k+1}+(1-\theta)\widehat{X}^{N,\theta}_k\Big)+\sigma\Delta W_k,\\
		\widehat{X}^{N,\theta}_0&=X_0,
	\end{split}\right.
\end{equation}
where $\Delta W_k:=W_{(k+1)h}-W_{kh}$, $k=0,1,\ldots,N-1$. 
We define the continuous version of the $\theta$ method \eqref{stochastic theta scheme} as
\begin{align*}
	X^{N,\theta}_t
	&=X_0+\int_{0}^{t}b(Z^{N,\theta}_s)\dd s+\int_{0}^{t}\sigma\dd W_s,\quad t\in[0,T],
\end{align*}
where $Z^{N,\theta}_s:=\theta X^{N,\theta}_{R(s)}+(1-\theta)X^{N,\theta}_{L(s)}$, $L(s):=\lfloor s/h\rfloor h$ and $R(s):=\lceil s/h\rceil h$. 
Here, $\lfloor\cdot\rfloor$ and $\lceil \cdot\rceil$ represent the floor and ceiling functions, respectively.
It is clear that $X_{kh}^{N,\theta}=\widehat{X}^{N,\theta}_k$ for $k\in\{0,1,\ldots,N\}$.

\begin{rema}\label{rema2}
	Based on the fundamental convergence theorem (see \cite[Theorem 1.1.1]{Milstein2004}), one can obtain the following estimate
	\begin{equation}\label{first_order}
	\sup_{0\leq t\le T}\big(\mathbb{E}\big[\|X_t-X^{N,\theta}_t\|^{2p}\big]\big)^{\frac{1}{2p}}
	\leq C_ph,\quad \forall \,p\in\mathbb{N}_+.
\end{equation}
\end{rema}
The inequality \eqref{first_order} reveals that the error of the $\theta$-method for \eqref{SHE} has first-order convergence of accuracy, which motivates us to define the normalized error process 
\begin{equation}\label{wss}
	U^{N,\theta}_t:=N(X_t-X^{N,\theta}_t),\quad t\in[0,T].
\end{equation}
To obtain the asymptotic error distribution of \eqref{wss}, we introduce an auxiliary process $\widetilde{U}^{N,\theta}=\{\widetilde{U}^{N,\theta},t\in[0,T]\}$ via
\begin{align}\label{eqfuzhu}
		\widetilde{U}^{N,\theta}_t=&\int_{0}^{t}Db(X_s)\widetilde{U}^{N,\theta}_s\dd s+N\int_{0}^{t}T^{N,\theta}_s Db(X^{N,\theta}_s)b(X^{N,\theta}_s)\dd s\\\notag
	&+N\int_{0}^{t}Db(X^{N,\theta}_{L(s)})(\sigma S^{N,\theta}_s)\dd s\\\notag
	&+N\int_{0}^{t}D^2b(X^{N,\theta}_{L(s)})\big(
	\sigma(W_s-W_{L(s)})
	,\sigma S^{N,\theta}_s\big)\dd s\\\notag
	&-\frac{N}{2}\int_{0}^{t}D^2b\big(X^{N,\theta}_{L(s)}\big)\big(\sigma S^{N,\theta}_s,\sigma S^{N,\theta}_s\big)\dd s
	=:\sum_{i=0}^{4}I_i^{N,\theta}(t)
\end{align}
for $t\in[0,T]$,
where
\begin{equation}\label{TandS}
	\begin{split}
		T^{N,\theta}_s:&=(1-\theta)(s-L(s))-\theta(R(s)-s),\\
	S^{N,\theta}_s:&=(1-\theta)\big(W_s-W_{L(s)}\big)-\theta\big(W_{R(s)}-W_s\big).
	\end{split}
\end{equation}

The following lemma shows that the normalized error process has the same limit distribution as the auxiliary process if either of them converges in distribution.
\begin{lemm}\label{dzm1}
	Let Assumption \ref{ass1} hold. Then for any $\theta\in[0,1]$, we have
	\begin{align*}
		\lim_{N\rightarrow\infty}\mathbb{E}\Big[\sup_{0\leq t\le T}\|U^{N,\theta}_t-\widetilde{U}^{N,\theta}_t\|^2\Big]=0.
	\end{align*}
\end{lemm}
\begin{proof}
	By the mean value theorem,
	\begin{equation}\label{eqwuchaguocheng}
		\begin{split}
			 U^{N,\theta}_t
			&=N\int_{0}^{t}b(X_s)-b(X^{N,\theta}_s)\dd s+N\int_{0}^{t}b(X^{N,\theta}_s)-b(Z^{N,\theta}_s)\dd s\\
			&=N\int_{0}^{t}Db(X_s)(X_s-X^{N,\theta}_s)\dd s\\
			&\quad -N\int_{0}^{t}\int_{0}^{1}(1-\xi)D^2b_\xi(X_s,X^{N,\theta}_s)(X^{N,\theta}_s-X_s,X^{N,\theta}_s-X_s)\dd\xi\dd s\\
			&\quad+N\int_{0}^{t}Db(X^{N,\theta}_s)(X^{N,\theta}_s-Z^{N,\theta}_s)\dd s\\
			&\quad-N\int_{0}^{t}\int_{0}^{1}(1-\xi)D^2b_\xi(X^{N,\theta}_s,Z^{N,\theta}_s)(Z^{N,\theta}_s-X^{N,\theta}_s,Z^{N,\theta}_s-X^{N,\theta}_s)\dd\xi\dd s\\
			&=:\sum_{i=0}^{3}R_i^{N,\theta}(t),
		\end{split}
	\end{equation}
	where $D^2b_\xi(Y,\widetilde{Y}):=D^2b((1-\xi)Y+\xi\widetilde{Y})$ for $\xi\in[0,1]$ and any two $n$-dimensional vectors $Y$ and $\widetilde{Y}$.	
 In view of $X^{N,\theta}_s-Z^{N,\theta}_s
	=b(Z^{N,\theta}_s)T^{N,\theta}_s+\sigma S^{N,\theta}_s$,
we can further obtain
	\begin{align*}
	R_2^{N,\theta}(t)
		&=N\int_{0}^{t}T^{N,\theta}_s Db(X^{N,\theta}_s)b(X^{N,\theta}_s)\dd s+N\int_{0}^{t}T^{N,\theta}_s Db(X^{N,\theta}_s)\big(b(Z^{N,\theta}_s)-b(X^{N,\theta}_s)\big)\dd s\\
		&\quad+N\int_{0}^{t}Db(X^{N,\theta}_{L(s)})(\sigma S^{N,\theta}_s)\dd s+N\int_{0}^{t}D^2b(X^{N,\theta}_{L(s)})\big(
			\sigma(W_s-W_{L(s)})
			,\sigma S^{N,\theta}_s\big)\dd s\\
		&\quad+N\int_{0}^{t}D^2b(X^{N,\theta}_{L(s)})\Big(
				\int_{L(s)}^{s}b(Z^{N,\theta}_s)\dd r
				,\sigma S^{N,\theta}_s\Big)\dd s\\
		&\quad+N\int_{0}^{t}\int_{0}^{1}\big(D^2b_\xi(X^{N,\theta}_{L(s)},X^{N,\theta}_s)-D^2b(X^{N,\theta}_{L(s)})\big)\big(X^{N,\theta}_s-X^{N,\theta}_{L(s)},\sigma S^{N,\theta}_s\big)\dd\xi\dd s,
	\end{align*}
	and
\begin{align*}
	R_3^{N,\theta}(t)
	=&-N\int_{0}^{t}(T^{N,\theta}_s)^2\int_{0}^{1}(1-\xi)D^2b_\xi(X^{N,\theta}_s,Z^{N,\theta}_s)\big(b(Z^{N,\theta}_s),b(Z^{N,\theta}_s)\big)\dd\xi\dd s\\
	&-2N\int_{0}^{t}T^{N,\theta}_s\int_{0}^{1}(1-\xi)D^2b_\xi(X^{N,\theta}_s,Z^{N,\theta}_s)\big(b(Z^{N,\theta}_s),\sigma S^{N,\theta}_s\big)\dd\xi\dd s\\
	&+N\int_{0}^{t}\int_{0}^{1}(1-\xi)\big(D^2b\big(X^{N,\theta}_{L(s)}\big)-D^2b_\xi\big(X^{N,\theta}_s,Z^{N,\theta}_s\big)\big)\big(\sigma S^{N,\theta}_s,\sigma S^{N,\theta}_s\big)\dd\xi\dd s\\
	&-N\int_{0}^{t}\int_{0}^{1}(1-\xi)D^2b\big(X^{N,\theta}_{L(s)}\big)\big(\sigma S^{N,\theta}_s,\sigma S^{N,\theta}_s\big)\dd\xi\dd s.
\end{align*}
It follows from \eqref{eqfuzhu} and \eqref{eqwuchaguocheng} that
\begin{align}\label{smzw}
	&\quad U^{N,\theta}_t-\widetilde{U}^{N,\theta}_t\notag\\
	&=\int_{0}^{t}Db(X_s)\big(U^{N,\theta}_s-\widetilde{U}^{N,\theta}_s\big)\dd s+R_1^{N,\theta}(t)\notag\\
	&\quad+N\int_{0}^{t}T^{N,\theta}_s Db(X^{N,\theta}_s)\big(b(Z^{N,\theta}_s)-b(X^{N,\theta}_s)\big)\dd s\notag\\
		&\quad+N\int_{0}^{t}D^2b(X^{N,\theta}_{L(s)})\Big(
	\int_{L(s)}^{s}b(Z^{N,\theta}_s)\dd r
	,\sigma S^{N,\theta}_s\Big)\dd s\\
	&\quad+N\int_{0}^{t}\int_{0}^{1}\big(D^2b_\xi(X^{N,\theta}_{L(s)},X^{N,\theta}_s)-D^2b(X^{N,\theta}_{L(s)})\big)\big(X^{N,\theta}_s-X^{N,\theta}_{L(s)},\sigma S^{N,\theta}_s\big)\dd\xi\dd s\notag\\
	&\quad-N\int_{0}^{t}(T^{N,\theta}_s)^2\int_{0}^{1}(1-\xi)D^2b_\xi(X^{N,\theta}_s,Z^{N,\theta}_s)\Big(b(Z^{N,\theta}_s),b(Z^{N,\theta}_s)\Big)\dd\xi\dd s\notag\\
	&\quad-2N\int_{0}^{t}T^{N,\theta}_s\int_{0}^{1}(1-\xi)D^2b_\xi(X^{N,\theta}_s,Z^{N,\theta}_s)\big(b(Z^{N,\theta}_s),\sigma S^{N,\theta}_s\big)\dd\xi\dd s\notag\\
	&\quad+N\int_{0}^{t}\int_{0}^{1}(1-\xi)\big(D^2b\big(X^{N,\theta}_{L(s)}\big)-D^2b_\xi\big(X^{N,\theta}_s,Z^{N,\theta}_s\big)\big)\big(\sigma S^{N,\theta}_s,\sigma S^{N,\theta}_s\big)\dd\xi\dd s.\notag
\end{align}
In view of \eqref{first_order}, we have
\begin{align*}
	\mathbb{E}\Big[\sup_{0\leq t\le T}\|R_1^{N,\theta}(t)\|^2\Big]&\leq CN^2\mathbb{E}\bigg[\bigg(\int_{0}^{T}\|X_s^{N,\theta}-X_s\|^2\dd s\bigg)^2\bigg]\leq Ch^2.
\end{align*}
As other terms in \eqref{smzw} can be similarly estimated, 
it holds that
	\begin{align*}
		&\quad\mathbb{E}\Big[\sup_{0\leq s\le t}\big\|U^{N,\theta}_s-\widetilde{U}^{N,\theta}_s\big\|^2\Big]\\
		&\leq C\mathbb{E}\Big[\sup_{0\leq s\le t}\Big\|\int_{0}^{s}Db(X_r)\big(\widetilde{U}^{N,\theta}_r-U^{N,\theta}_r\big)\dd r\Big\|^2\Big]+Ch\\
		&\leq C\int_{0}^{t}\mathbb{E}\Big[\sup_{0\leq r\leq s}\|\widetilde{U}^{N,\theta}_r-U^{N,\theta}_r\|^2\Big]\dd s+Ch,
	\end{align*}
	which, together with the Gronwall inequality, finishes the proof.
\end{proof}

Next, we present the convergence of $I^{N,\theta}_i$, $i=1,2,3,4$, defined in \eqref{eqfuzhu} as the discretization parameter $N$ goes to infinity. 
In the sequel, we denote by $``\stackrel{d}{\Rightarrow}"$ the convergence in distribution for random variables.

 \begin{lemm}\label{rss}
 	Let Assumption \ref{ass1} hold. Then for any $t\in[0,T]$, $I_i^{N,\theta}(t)$ converges to $I_i^{\theta}(t)$ in $\mathbb{L}^2(\Omega;\mathbb{R}^{2d})$ as $N\rightarrow\infty$ for $i=1,3,4$. Besides, $I_2^{N,\theta}\stackrel{d}{\Rightarrow}I_2^{\theta}$ in $\mathbb{C}([0,T];\mathbb{R}^{2d})$ as $N\rightarrow\infty$. Here, $I^{\theta}_i,i=1,2,3,4$ are defined as
 	\begin{align*}
 	I_1^{\theta}(t)&=\frac{(1-2\theta)T}{2}\int_{0}^{t}Db(X_s)b(X_s)\dd s,\\
 		I_2^{\theta}(t)&=\frac{(1-2\theta)T}{2}\int_{0}^{t}Db(X_s)\sigma\dd W_s+\frac{\sqrt3}{6}T\int_{0}^{t}Db(X_s)\sigma\dd\widetilde{W}_s,\\
 		I_{3}^{\theta}(t)&=\frac{(1-\theta)T}{2}\sum_{k=1}^{m}\int_{0}^{t}D^2b(X_s)(\sigma_{\cdot,k},\sigma_{\cdot,k})\dd s,\\
 	I_4^{\theta}(t)&=-\frac{(1-2\theta+2\theta^2)T}{4}\sum_{k=1}^{m}\int_{0}^{t}D^2b(X(s))(\sigma_{\cdot,k},\sigma_{\cdot,k})\dd s,
 	\end{align*}
 	where $\{\widetilde{W}_t,t\in[0,T]\}$ is an $m$-dimensional standard Brownian motion independent of $\{W_t,t\in[0,T]\}$.
 \end{lemm}

\begin{proof}
	We estimate the four terms $\{I_i^{N,\theta}(t),i=1,2,3,4\}$ separately.
	
	\textbf{Estimate of $I_1^{N,\theta}(t)$}.
	Since $NT^{N,\theta}_s=T\big(\frac{Ns}{T}-\lfloor\frac{Ns}{T}\rfloor\big)-\theta T$, we have 
\begin{align*}
	I_1^{N,\theta}(t)
	&=T\int_{0}^{t}Db(X^{N,\theta}_s)b(X^{N,\theta}_s)\Big\{\frac{Ns}{T}-\lfloor\frac{Ns}{T}\rfloor\Big\}\dd s-\theta T\int_{0}^{t}Db(X^{N,\theta}_s)b(X^{N,\theta}_s)\dd s.
\end{align*}
Due to the fact that
 $\lim_{N\rightarrow\infty}\mathbb{E}\int_{0}^{t}\|Db(X^{N,\theta}_s)-Db(X_s)\|^2\dd s=0$ (see Remark \ref{rema2}), applying \cite[Proposition 4.2]{hong2024asymptoticerrordistributionaccelerated} yields that $\lim_{N\rightarrow\infty}I_1^{N,\theta}(t)=I_1^\theta(t)$ in $\mathbb{L}^2(\Omega;\mathbb{R}^{2d})$.
 
\textbf{Estimate of $I_2^{N,\theta}(t)$}.
By \eqref{TandS} and the stochastic Fubini theorem, it holds that
\begin{align*}
	I_2^{N,\theta}(t)&=(1-\theta)N\int_{0}^{t}\int_{L(s)}^{s}Db(X^{N,\theta}_{L(s)})\sigma\dd W_r\dd s-\theta N\int_{0}^{t}\int_{s}^{R(s)}Db(X^{N,\theta}_{L(s)})\sigma\dd W_r\dd s\\
	&=\widehat{I}_2^{N,\theta}(t)+\widetilde{I}_2^{N,\theta}(t),
\end{align*}	
for any $t\in[0,T]$, where	
\begin{align*}
	\widehat{I}_2^{N,\theta}(t):=(1-\theta)N\int_{0}^{t}Db(X^{N,\theta}_{L(r)})\sigma(R(r)-r)\dd W_r-\theta N\int_{0}^{t}Db(X^{N,\theta}_{L(r)})\sigma(r-L(r))\dd W_r,\\
	\widetilde{I}_2^{N,\theta}(t):=-(1-\theta)N\int_{L(t)}^{t}Db(X^{N,\theta}_{L(r)})\sigma(R(t)-t)\dd W_r-\theta N\int_{t}^{R(t)}Db(X^{N,\theta}_{L(r)})\sigma(t-L(t))\dd W_r.
\end{align*}
Denoting by $\widehat{I}_{2}^{N,\theta,k}(t)$ the $k$th component of $\widehat{I}_2^{N,\theta}(t)$, we have
\begin{align*}
	\widehat{I}_{2}^{N,\theta,k}(t)=T\sum_{i=1}^{n}\sum_{j=1}^{m}\sigma_{i,j}\int_{0}^{t}\partial_ib^k(X^{N,\theta}_{L(s)})\Big\{\lfloor\frac{Ns}{T}\rfloor-\frac{Ns}{T}+1-\theta\Big\}\dd W_s^j,\quad k=1,\ldots,n.
\end{align*}
For $\nu\in\{1,\ldots,m\}$ and $\mu\in\{1,\ldots,n\}$, the cross variation process between $\widehat{I}_{2}^{N,\theta,\mu}$ and $W^\nu$ is
\begin{align*}
	\langle \widehat{I}_{2}^{N,\theta,\mu},W^\nu\rangle_t
	&=\sum_{i=1}^{n}\int_{0}^{t} \partial_ib^\mu(X^{N,\theta}_{L(s)})\sigma_{i,\nu}T\Big\{\lfloor\frac{Ns}{T}\rfloor-\frac{Ns}{T}+1-\theta\Big\}\dd s,
\end{align*}
which combined with \cite[Proposition 4.2]{hong2024asymptoticerrordistributionaccelerated} 
 and Remark \ref{rema2} leads to
\begin{align}\label{qww}
	\lim_{N\rightarrow\infty}\langle \widehat{I}_{2}^{N,\theta,\mu},W^\nu\rangle_t=\frac{(1-2\theta)T}{2}\int_{0}^{t}\sum_{i=1}^{n}\partial_ib^\mu(X_s)\sigma_{i,\nu}\dd s \quad\text{in }\mathbb{L}^2(\Omega;\mathbb{R}).
\end{align}
Similarly, it holds that for any $\mu,\widetilde{\mu}\in\{1,\ldots,n\}$,
\begin{equation}\label{wjz}
\lim_{N\rightarrow\infty}\langle \widehat{I}_{2}^{N,\theta,\mu},\widehat{I}_{2}^{N,\theta,\widetilde{\mu}}\rangle_t
	=\frac{\theta^3+(1-\theta)^3}{3}T^2\sum_{i_1,i_2=1}^{n}\sum_{j=1}^{m}\int_{0}^{t}\partial_{i_1}b^\mu(X_s)\partial_{i_2}b^{\widetilde{\mu}}(X_s)\sigma_{i_1,j}\sigma_{i_2,j}\dd s.
\end{equation}
Based on \eqref{qww} and \eqref{wjz}, applying \cite[Theorem 4-1, Proposition 1-4]{Jacod1997} produces that
$\widehat{I}_2^{N,\theta}\stackrel{d}{\Rightarrow}I_2^\theta$ in $\mathbb{C}([0,T];\mathbb{R}^{2d})$. This, together with the fact that
$\widetilde{I}_2^{N,\theta}$ converges to $0$ in $L^2(\Omega;\mathbb{C}([0,T];\mathbb{R}^{2d}))$, shows that $I_2^{N,\theta}\stackrel{d}{\Rightarrow}I_2^\theta$ in $\mathbb{C}([0,T];\mathbb{R}^{2d})$.

\textbf{Estimate of $I_3^{N,\theta}(t)$}.
In view of \eqref{TandS}, we divide
	$I_3^{N,\theta}(t)
	=I^{N,\theta}_{3,1}(t)+I^{N,\theta}_{3,2}(t)$,
where
\begin{gather*}
	I^{N,\theta}_{3,1}(t):=(1-\theta)N\int_{0}^{t}D^2b(X^{N,\theta}_{L(s)})\big(\sigma(W_s-W_{L(s)}),\sigma(W_s-W_{L(s)})\big)\dd s,\\
	I^{N,\theta}_{3,2}(t):=-\theta N\int_{0}^{t}D^2b(X^{N,\theta}_{L(s)})\big(\sigma(W_s-W_{L(s)}),\sigma(W_{R(s)}-W_s)\big)\dd s.
\end{gather*}
Notice that
\begin{align*}
	&\quad I_{3,1}^{N,\theta}(t)=(1-\theta)\sum_{i_1,j_1=1}^{n}\sum_{k=1}^{m}N\int_{0}^{t}\partial_{i_1,j_1}b(X^{N,\theta}_{L(s)})\sigma_{i_1,k}\sigma_{j_1,k}(W_s^{k}-W^{k}_{L(s)})^2\dd s\\
	&\quad+(1-\theta)\sum_{i_1,j_1=1}^{n}\sum_{\substack{i_2,j_2=1\\i_2\neq j_2}}^{m}N\int_{0}^{t}\partial_{i_1,j_1}b(X^{N,\theta}_{L(s)})\sigma_{i_1,i_2}\sigma_{j_1,j_2}(W_s^{i_2}-W^{i_2}_{L(s)})(W^{j_2}_s-W^{j_2}_{L(s)})\dd s\\
	&=:\widetilde{I}_{3,1}^{N,\theta}(t)+\widehat{I}_{3,1}^{N,\theta}(t).
\end{align*}
By the relation $\dd(W_t^iW_t^j)=W_t^i\dd W_t^j+W_t^j\dd W_t^i+\delta_{i,j}\dd t$ for $i,j\in\{1,\ldots,m\}$ and the stochastic Fubini theorem,
we obtain 
\begin{align*}
	\widetilde{I}_{3,1}^{N,\theta}(t)&=2(1-\theta)N\sum_{i_1,j_1=1}^{n}\sum_{k=1}^{m}\int_{0}^{t}\partial_{i_1,j_1}b(X^{N,\theta}_{L(r)})\sigma_{i_1,k}\sigma_{j_1,k}(W_r^k-W^k_{L(r)})(R(r)\wedge t-r)\dd W_r^k\\
	&\quad+(1-\theta)T\sum_{i_1,j_1=1}^{n}\sum_{k=1}^{m}\int_{0}^{t}\partial_{i_1,j_1}b(X^{N,\theta}_{L(s)})\sigma_{i_1,k}\sigma_{j_1,k}\Big(\frac{Ns}{T}-\lfloor\frac{Ns}{T}\rfloor\Big)\dd s\\
	&=:I_{3,1,1}^{N,\theta}(t)+I_{3,1,2}^{N,\theta}(t)
\end{align*}
and
\begin{align*}
	&\quad\widehat{I}_{3,1}^{N,\theta}(t)\\&=(1-\theta)N\sum_{i_1,j_1=1}^{n}\sum_{\substack{i_2,j_2=1\\i_2\neq j_2}}^{m}\int_{0}^{t}\partial_{i_1,j_1}b(X^{N,\theta}_{L(r)})\sigma_{i_1,i_2}\sigma_{j_1,j_2}(R(r)\wedge t-r)(W_r^{i_2}-W_{L(r)}^{i_2})\dd W_r^{j_2}\\
	&\quad+(1-\theta)N\sum_{i_1,j_1=1}^{n}\sum_{\substack{i_2,j_2=1\\i_2\neq j_2}}^{m}\int_{0}^{t}\partial_{i_1,j_1}b(X^{N,\theta}_{L(r)})\sigma_{i_1,i_2}\sigma_{j_1,j_2}(R(r)\wedge t-r)(W_r^{j_2}-W_{L(r)}^{j_2})\dd W_r^{i_2}\\
	&=:I_{3,1,3}^{N,\theta}(t)+I_{3,1,4}^{N,\theta}(t).
\end{align*}
Analogous to the estimate of $I_1^{N,\theta}(t)$, it can be shown that 
\begin{align*}
	\lim_{N\rightarrow\infty}I_{3,1,2}^{N,\theta}(t)=\frac{1-\theta}{2}T\sum_{k=1}^{m}\int_{0}^{t}D^2b(X(s))(\sigma_{\cdot,k},\sigma_{\cdot,k})\dd s=I_{3}^{\theta}(t)\quad\text{ in }\mathbb{L}^2(\Omega;\mathbb{R}^{2d}).
\end{align*}
 Moreover, $\lim_{N\rightarrow\infty}I_{3,1,i}^{N,\theta}=0$ in $\mathbb{L}^2(\Omega;\mathbb{C}([0,T];\mathbb{R}^{2d}))$ for $i\in\{1,3,4\}$.
Hence, for any fixed $t\in[0,T]$,
\begin{align}\label{eq:I3t}
	\lim_{N\rightarrow\infty}I_{3,1}^{N,\theta}(t)=I_{3}^{\theta}(t)\quad\text{ in }\mathbb{L}^2(\Omega;\mathbb{R}^{2d}).
\end{align}
By the stochastic Fubini theorem,
\begin{align*}
	I_{3,2}^{N,\theta}(t)
	&=-\theta N\sum_{i_1,i_2=1}^{n}\sum_{i_2,j_2=1}^{m}\int_{0}^{t}\int_{L(r)}^{r}\partial_{i_1,j_1}b(X^{N,\theta}_{L(r)})\sigma_{i_1,i_2}\sigma_{j_1,j_2}\big(W_s^{i_2}-W_{L(r)}^{i_2}\big)\dd s\dd W_r^{j_2}\\
	&\quad-\theta N\sum_{i_1,i_2=1}^{n}\sum_{i_2,j_2=1}^{m}\int_{t}^{R(t)}\int_{L(t)}^{t}\partial_{i_1,j_1}b(X^{N,\theta}_{L(t)})\sigma_{i_1,i_2}\sigma_{j_1,j_2}\big(W_s^{i_2}-W_{L(t)}^{i_2}\big)\dd s\dd W_r^{j_2},
\end{align*}
from which we obtain that $\lim_{N\rightarrow\infty}I_{3,2}^{N,\theta}=0$ in $\mathbb{L}^2(\Omega;\mathbb{C}([0,T];\mathbb{R}^{2d}))$. This along with \eqref{eq:I3t} proves the convergence in $\mathbb{L}^2(\Omega;\mathbb{R}^{2d})$ of $I_{3}^{N,\theta}(t)$ to $I_{3}^{\theta}(t)$ for any fixed $t\in[0,T]$.

\textbf{Estimate of $I_4^{N,\theta}(t)$}.
For the term $I_4^{N,\theta}(t)$, we have
\begin{align*}
	I_4^{N,\theta}(t)
	&=-\frac{N}{2}(1-\theta)^2\int_{0}^{t}D^2b(X^{N,\theta}_{L(s)})\big(\sigma(W_s-W_{L(s)}),\sigma(W_s-W_{L(s)})\big)\dd s\\
	&\quad+N(1-\theta)\theta\int_{0}^{t}D^2b(X^{N,\theta}_{L(s)})\big(\sigma(W_s-W_{L(s)}),\sigma(W_{R(s)}-W_s)\big)\dd s\\
	&\quad-\frac{N}{2}\theta^2\int_{0}^{t}D^2b(X^{N,\theta}_{L(s)})\big(\sigma(W_{R(s)}-W_s),\sigma(W_{R(s)}-W_s)\big)\dd s.
\end{align*}
The first term and third term on the right hand side can be estimated similarly to $I_{3,1}^{N,\theta}(t)$; while the second term on the right hand side converges to $0$ (see the estimate of $I_{3,2}^{N,\theta}(t)$). As a result, we derive that
\begin{align*}
	\lim_{N\rightarrow\infty}I_4^{N,\theta}(t)=-\frac{(1-2\theta+2\theta^2)T}{4}\sum_{k=1}^{m}\int_{0}^{t}D^2b(X_s)(\sigma_{\cdot,k},\sigma_{\cdot,k})\dd s\quad \text{in }\mathbb{L}^2(\Omega;\mathbb{R}^{2d}).
\end{align*}
The proof is completed.
\end{proof}

\begin{lemm}\label{rss1}
Let Assumption \ref{ass1} hold and $\theta\in[0,1]$.	Then the sequence $$\{(\widetilde{U}^{N,\theta},I_0^{N,\theta},I_1^{N,\theta},I_2^{N,\theta},\\I_3^{N,\theta},I_4^{N,\theta},W,\widetilde{W},X)\}_{N\geq 1}$$ is tight in $\mathbb{C}([0,T];\mathbb{R}^{2d})^{\otimes9}$.
\end{lemm}
\begin{proof}
	It suffices to show that each component of $\{(\widetilde{U}^{N,\theta},I_0^{N,\theta},I_1^{N,\theta},I_2^{N,\theta},I_3^{N,\theta},I_4^{N,\theta},\\W,\widetilde{W},X)\}_{N\geq 1}$ is tight.
	 By the Gronwall inequality, we can obtain from \eqref{eqfuzhu} that there exists some constant $C$ independent of $N$ such that
	 \begin{align}
	 \sup_{0\leq t\leq T}\|\widetilde{U}^{N,\theta}_r\|_{\mathbb{L}^2(\Omega;\mathbb{R}^{2d})}\leq C.
	 \end{align}	 
Thus, it holds that for any $T\ge t\ge s\ge0$,
	\begin{align*}
		\|I_0^{N,\theta}(t)-I_0^{N,\theta}(s)\|_{\mathbb{L}^2(\Omega;\mathbb{R}^{2d})}
		&\leq \int_{s}^{t}\Big\|Db(X_r)\widetilde{U}^{N,\theta}_r\Big\|_{\mathbb{L}^2(\Omega;\mathbb{R}^{2d})}\dd r
		\leq C(t-s).
	\end{align*}
	By the Kolmogorov continuity theorem (see \cite[Chapter I, Theorem 2.1]{Daniel1999}), we have
	\begin{align}\label{eq:Kol}
		\sup_{N\geq 1}\Big\|\sup_{t\neq s}\frac{\|I_0^{N,\theta}(t)-I_0^{N,\theta}(s)\|}{|t-s|^{\frac{1}{4}}}\Big\|_{\mathbb{L}^2(\Omega;\mathbb{R})}\leq C.
	\end{align}
	Since $I^{N,\theta}_0=0$, the sequence $\{I_0^{N,\theta}\}_{N\ge1}$ is uniformly bounded in $\mathbb{L}^2(\Omega;\mathbb{C}([0,T];\mathbb{R}^{2d}))$,
	which in combination with \eqref{eq:Kol} implies that
	\begin{align*}
		\sup_{N\geq 1}\mathbb{E}\Big[\|I_0^{N,\theta}\|_{\mathbb{C}^{1/4}([0,T];\mathbb{R}^{2d})}\Big]\leq C.
	\end{align*}
	For $L>0$, denote $B_L:=\{f\in\mathbb{C}([0,T];\mathbb{R}^{2d}):\|f\|_{\mathbb{C}^{1/4}([0,T];\mathbb{R}^{2d})}\leq L\}$, which is pre-compact set of $\mathbb{C}([0,T];\mathbb{R}^{2d})$. Then by the Markov inequality,
	\begin{align*}
		\mathbb{P}(I_0^{N,\theta}\notin \overline{B_L})\leq \mathbb{P}(I_0^{N,\theta}\notin B_L)\leq \frac{1}{L}\mathbb{E}\big[\|I_0^{N,\theta}\|_{\mathbb{C}^{1/4}([0,T];\mathbb{R}^{2d})}\big]\leq \frac{C}{L}\rightarrow0,\quad\text{as }L\rightarrow\infty,
	\end{align*}
which yields the tightness of $\{I_0^{N,\theta}\}_{N\geq 1}$ in $\mathbb{C}([0,T];\mathbb{R}^{2d})$. Similarly, one can prove the tightness of $\{I_i^{N,\theta}\}_{N\geq 1}$ in $\mathbb{C}([0,T];\mathbb{R}^{2d})$ by verifying $\|I_i^{N,\theta}(t)-I_i^{N,\theta}(s)\|_{\mathbb{L}^2(\Omega;\mathbb{R}^{2d})}\leq C|t-s|$ for $i=1,3,4$. By Lemma \ref{rss}, $I_2^{N,\theta}\stackrel{d}{\Rightarrow}I_2^\theta$ in $\mathbb{C}([0,T];\mathbb{R}^{2d})$ as $N\rightarrow\infty$, which implies the tightness of $\{I_2^{N,\theta}\}_{N\geq 1}$ in $\mathbb{C}([0,T];\mathbb{R}^{2d})$ due to Prokhorov's theorem (see \cite[Theorem 13.29]{KlenkeAchim2020}).
	
	By the Burkholder--Davis--Gundy inequality, for any $p\geq 1$,
	\begin{align*}
		\|X_t-X_s\|_{\mathbb{L}^p(\Omega;\mathbb{R}^{2d})}+\|W_t-W_s\|_{\mathbb{L}^p(\Omega;\mathbb{R}^{2d})}+\|\widetilde{W}_t-\widetilde{W}_s\|_{\mathbb{L}^p(\Omega;\mathbb{R}^{2d})}\leq C|t-s|^{1/2}.
	\end{align*}
	Following the argument for the tightness of $\{I_0^{N,\theta}\}_{N\geq1 }$, we obtain the tightness of $X$, $W$ and $\widetilde{W}$ in $\mathbb{C}([0,T];\mathbb{R}^{2d})$. Thus the proof is finished.
\end{proof}

Combining Lemmas \ref{dzm1}, \ref{rss}, and \ref{rss1}, we can conclude the asymptotic error distribution of the $\theta$ method \eqref{stochastic theta scheme} for \eqref{SHE}.
\begin{prop}\label{smz}
	Let Assumption \ref{ass1} hold and $\theta\in[0,1]$. Then for any $t\in[0,T]$, $U^{N,\theta}_t$ converges to $U_t^\theta$ in the sense of distribution. Here, $U_t^\theta$ satisfies the following equation
	\begin{align*}
		U_t^\theta=&\int_{0}^{t}Db(X_s)U_s^\theta\dd s+\frac{(1-2\theta)T}{2}\int_{0}^{t}Db(X_s)b(X_s)\dd s\\
		&+\frac{(1-2\theta)T}{2}\int_{0}^{t}Db(X_s)\sigma\dd W_s+\frac{\sqrt 3}{6}T\int_{0}^{t}Db(X_s)\sigma\dd\widetilde{W}_s\\
		&+\frac{(1-2\theta^2)T}{4}\sum_{k=1}^{m}\int_{0}^{t}D^2b(X_s)(\sigma_{\cdot,k},\sigma_{\cdot,k})\dd s,
	\end{align*}
	where $\{\widetilde{W}_t,t\in[0,T]\}$ is an $m$-dimensional standard Brownian motion independent of $\{W_t,t\in[0,T]\}$.
\end{prop}
\begin{proof}
	The proof is similar to that of \cite[Lemma 3.6]{hong2024asymptoticerrordistributionaccelerated}, and thus is omitted.
\end{proof}

Proposition \ref{smz} also holds for general stochastic differential equations with constant diffusion term $\sigma$. It suggests that the strong convergence order $1$ is sharp for the $\theta$ method applied to nonlinear stochastic differential equations with additive noise.

\begin{theo}\label{daa}
Let Assumption \ref{ass1} hold and $\theta\in[0,1]$.
Assume that there exists $R>0$ such that the Hessian matrix $\nabla^2 H$ is uniformly positive definite for all $x\in\mathbb{R}^{2d}$ with $\|x\|>R$. Then there exist some constants $C_i$, $i=0,1,2,3$ independent of $\theta$ and $T$ such that
for any $t\in[0,T]$,
\begin{align}\label{eq:thm}
		\mathbb{E}\left[\|U_t^\theta\|^2\right]&\le f_\theta(T):=e^{C_0t}\left(C_1(1-2\theta)^2T^4+C_2(\theta^2-\theta+\frac13) T^3+C_3(1-2\theta^2)^2T^2\right).
	\end{align}
\end{theo}
\begin{proof}
In this proof, we denote by $\textup{K}_1, \textup{K}_2,\ldots,$ the generic constants that may dependent on $H$ and $\sigma$, but independent of $\theta$ and $T$.
Notice that $\widehat{W}:=(\theta^2-\theta+\frac13)^{-\frac12}(\frac{(1-2\theta)}{2}W+\frac{\sqrt 3}{6}\widetilde{W})$ is a standard Brownian motion with respect to the filtration generated by $W$ and $\widetilde{W}$, and
\begin{align*}
\frac{(1-2\theta)T}{2}\int_{0}^{t}Db(X_s)\sigma\dd W_s+\frac{\sqrt 3}{6}T\int_{0}^{t}Db(X_s)\sigma\dd\widetilde{W}_s\\
=\sqrt{\theta^2-\theta+\frac13}T\int_{0}^{t}Db(X_s)\sigma\dd \widehat{W}_s.
\end{align*}
Then by the It\^o formula and Proposition \ref{smz}, we have
	\begin{align*}
		\frac12\mathbb{E}\left[\|U_t^\theta\|^2\right]&=\mathbb{E}\int_{0}^{t}(U_s^\theta)^\top Db(X_s)U_s^\theta\dd s+\frac{(1-2\theta)T}{2}\mathbb{E}\int_{0}^{t}(U_s^\theta)^\top Db(X_s)b(X_s)\dd s\\
		&+\frac{(1-2\theta^2)T}{4}\sum_{k=1}^{m}\mathbb{E}\int_{0}^{t}(U_s^\theta)^\top D^2b(X_s)(\sigma_{\cdot,k},\sigma_{\cdot,k})\dd s\\
		&+(\theta^2-\theta+\frac13)T^2\mathbb{E}\int_{0}^{t}\textup{tr}(Db(X_s)\sigma\sigma^\top Db(X_s)^\top)\dd s.
	\end{align*}
	By the Young inequality and Assumption \ref{ass1}, for any $t\in[0,T]$,
		\begin{align}\label{eq:Utheta}\notag
		\mathbb{E}\left[\|U_t^\theta\|^2\right]&\le \textup{K}_0\int_{0}^{t}\mathbb{E}\left[\|U_s^\theta\|^2\right]\dd s+\textup{K}_1(1-2\theta)^2T^2\mathbb{E}\int_{0}^{t}\|b(X_s)\|^2\dd s\\
		&\quad+\textup{K}_2(1-2\theta^2)^2T^2+\textup{K}_3(\theta^2-\theta+\frac13) T^3.
	\end{align}
Utilizing the It\^o formula again as well as the anti-symmetry of $J$, for any $t\in[0,T]$,
$$\mathbb{E}\left[H(X_t)\right]=\frac12\int_0^t\textup{tr}(\nabla^2 H(X_s)\sigma\sigma^\top)\dd s\le \textup{K}_4t.$$
		In view of the assumption of $H$, there exists some $a_1, a_2>0$ for any $x\in\mathbb{R}^{2d}$,
		$$H(x)= H(0)+\nabla H(0)\cdot x+\frac{1}{2}\int_0^1(1-\eta) x^\top\nabla^2 H(\eta x)x\dd \eta\ge a_1\|x\|^2-a_2.$$
		 The linear growth of $b$ implies that for any $t\in[0,T]$,
		$$\mathbb{E}\left[\|b(X_t)\|^2\right]\le \textup{K}_5(1+\mathbb{E}\left[\|X_t\|^2\right])\le \textup{K}_5\Big(1+\frac{\textup{K}_4 t+a_2}{a_1}\Big).$$
		Plugging the above inequality into \eqref{eq:Utheta}, we obtain \eqref{eq:thm} from the Gronwall inequality.
\end{proof}

According to \eqref{eq:thm}, for any fixed sufficiently large $T>0$, 
\begin{equation*}
f_\frac12(T)=\min_{\theta\in[0,1]} f_\theta(T).
\end{equation*}
This implies that in the long-time computation, the error constant of the midpoint method may be smaller than that for $\theta$ method with $\theta\neq \frac12$.

\section{Error distribution of numerical methods for test equation}\label{chap3}

In this section, by taking the linear stochastic oscillator as the test equation, we study the asymptotic error distribution of several numerical methods.  We try to explain the
superiority of symplectic methods from the perspective of the asymptotic error
distribution.
In the sequel, let $R=\mathcal{O}(h^p),p\geq 0$, stand for $|R|\leq Ch^p$ for all sufficiently small $h>0$, where $C$ is independent of $h$. Let $K\sim T^p,p\geq 0$, stand for $\lim_{T\rightarrow \infty}K/T^p=C$, where $C$ is independent of $T$. We denote $\mathcal{N}(\mathbf{m},\Sigma)$ the normal distribution with mean $\mathbf{m}\in\mathbb{R}$ and variance $\Sigma> 0$ . 

By introducing $X^2_t:=\dot{X}^1_t$ and $X_t:=(X_t^1,X_t^2)^\top$, the linear stochastic oscillator
$\ddot{X}^1_t+X^1_t=\alpha \dot{W}_t$ ($\alpha>0$) can be rewritten as
\begin{align}\label{model_zz}
	\dd X_t
	=\begin{pmatrix}
	0&1\\
	-1&0	
\end{pmatrix}X_t\dd t+\alpha \begin{pmatrix}
	0\\
	1	
\end{pmatrix}\dd W_t,\quad t\in[0,T]
\end{align}
with initial data $X_0:=(X^1_0,X^2_0)^\top$.
Here, $W=\{W_t\}_{t\in[0,T]}$ denotes a one-dimensional standard Brownian motion defined on $(\Omega,\mathcal{F},\{\mathcal{F}_t\}_{t\geq 0},\mathbb{P})$. 
\begin{rema}
	The exact solution of \eqref{model_zz} is given by (see e.g., \cite{SAHD2004}) 
	\begin{equation}\label{exact_solution}
		\left\{ \begin{split}
			&X^1_t=X^1_0\cos t+X^2_0\sin t+\alpha\int_{0}^{t}\sin(t-s)\dd W_s,\\
			&X^2_t=-X^1_0\sin t+X^2_0\cos t+\alpha\int_{0}^{t}\cos(t-s)\dd W_s.
		\end{split}\right.
	\end{equation}
	Moreover, the symplectic structure of its phase flow is preserved (see, e.g., \cite{review2015}), i.e.,
	\begin{align*}
		\dd X^1_t\wedge\dd X^2_t=\dd X^1_0\wedge\dd X^2_0,\quad\forall\,t\geq0.
	\end{align*}
\end{rema}

We consider a general convergent numerical method for \eqref{model_zz} of the form
\begin{align}\label{general_numerical_zz}
	\begin{pmatrix}
		\widehat{X}_{k+1,1}^N \\
	\widehat{X}_{k+1,2}^N \end{pmatrix}
	=A\begin{pmatrix}
		\widehat{X}_{k,1}^N\\
		\widehat{X}_{k,2}^N \end{pmatrix}+\alpha b\Delta W_k:=\begin{pmatrix}
		a_{11}&a_{12}\\
		a_{21}&a_{22}
	\end{pmatrix} \begin{pmatrix}
		\widehat{X}_{k,1}^N\\
	\widehat{X}_{k,2}^N
	\end{pmatrix}+\alpha\begin{pmatrix}
		b_1\\b_2
	\end{pmatrix}\Delta W_k,
\end{align}
where $\Delta W_k:=W_{(k+1)h}-W_{kh}, k=0,1,\cdots,N-1$. 	

\begin{rema}
	The numerical method \eqref{general_numerical_zz} has first-order convergence of accuracy for \eqref{model_zz} if (see \cite[Theorem 4.1]{suijizhenzidapiancha})
	\begin{equation}\label{mzn3}
		|a_{11}-1|+|a_{22}-1|+|a_{12}-h|+|a_{21}+h|=\mathcal{O}(h^2),\quad|b_1|+|b_2-1|=\mathcal{O}(h).
	\end{equation} 
\end{rema}

A numerical method $\{(\widehat{X}_{k,1}^N,
\widehat{X}_{k,2}^N)\}_{k=0}^N$ for \eqref{model_zz} is called symplectic if 
\begin{align*}
	\dd \widehat{X}_{k+1,1}^N\wedge\dd \widehat{X}_{k+1,2}^N=\dd \widehat{X}_{k,1}^N\wedge\dd \widehat{X}_{k,2}^N,\qquad\forall \,k\in\{0,\ldots,N-1\}.
\end{align*}
Since $\{X^2_t\}_{t\in [0,T]}$ is the derivative of $\{X^1_t\}_{t\in[0,T]}$ and many physical observations (e.g., the mean position $\frac{1}{T}\int_{0}^{T}X^1_t\dd t$ and the mean velocity $\frac{X^1_T}{T}$) of \eqref{model_zz} depends on $\{X^1_t\}_{t\in[0,T] }$, we mainly consider the error $e_N:=\widehat{X}^N_{N,1}-X_T^1$. In terms of the numerical method \eqref{general_numerical_zz}, it follows from \cite{review2015} that 
\begin{equation}\label{numerical_solution}
	\begin{split}
	\widehat{X}^N_{N,1}&=(a_{11}\widehat{\alpha}_{N-1}+\widehat{\beta}_{N-1})X^1_0+a_{12}\widehat{\alpha}_{N-1}X^2_0\\
	&\quad+\alpha\sum_{k=0}^{N-1}\left(b_1\widehat{\alpha}_{N-1-k}+\gamma\widehat{\alpha}_{N-2-k}\right)\Delta W_k,
		\end{split}
\end{equation}
where $\gamma=a_{12}b_2-a_{22}b_1$,
	$\widehat{\alpha}_k=\left(\det(A)\right)^{\frac{k}{2}}\frac{\sin((k+1)\xi)}{\sin(\xi)}$, $\widehat{\beta}_k=-\left(\det(A)\right)^{\frac{k+1}{2}}\frac{\sin\left(k\xi\right)}{\sin(\xi)}$
for any integer $k$, with $\xi\in(0,\pi)$ satisfying
\begin{align}\label{jiaodu}
	\cos(\xi)=\frac{{\rm tr}(A)}{2\sqrt{\det(A)}},\qquad\sin(\xi)=\frac{\sqrt{4\det(A)-({\rm tr}(A))^2}}{2\sqrt{\det(A)}}.
\end{align}

Utilizing \eqref{exact_solution} and \eqref{numerical_solution}, it yields that
\begin{equation}\label{mlk}
	\begin{split}
		e_N
		&=[a_{11}\widehat{\alpha}_{N-1}+\widehat{\beta}_{N-1}-\cos(T)]X^1_0+[a_{12}\widehat{\alpha}_{N-1}-\sin(T)]X^2_0\\
		&\quad+\alpha\sum_{j=0}^{N-1}\int_{t_j}^{t_{j+1}}\left(b_1\widehat{\alpha}_{N-1-j}+\gamma\widehat{\alpha}_{N-2-j}-\sin(T-s)\right)\dd W_s.
	\end{split}
\end{equation}

Notice that for the linear problem \eqref{model_zz}, the error $e_N$ is a Gaussian random variable whose variance plays a crucial role in determining its distribution.
Hence, we further study the explicit expressions of the variances of the errors of general symplectic and non-symplectic methods. 

By \eqref{mlk} and It\^o isometry, one gets
	\begin{align*}
		\frac{{\rm Var}(e_N)}{\alpha^2}
		&=\sum_{j=0}^{N-1}\int_{t_j}^{t_{j+1}}\left(b_1\widehat{\alpha}_{N-1-j}+\gamma\widehat{\alpha}_{N-2-j}-\sin\left(T-s\right)\right)^2\dd s\\
		&=\sum_{j=0}^{N-1}\left(b_1\widehat{\alpha}_{N-1-j}+\gamma\widehat{\alpha}_{N-2-j}\right)^2h\\
		&\quad+\sum_{j=0}^{N-1}\left[\frac{h}{2}+\frac{1}{4}\sin\left(2T-2t_{j+1}\right)-\frac{1}{4}\sin\left(2T-2t_j\right)\right]\\
		&\quad-2\sum_{j=0}^{N-1}\left(b_1\widehat{\alpha}_{N-1-j}+\gamma\widehat{\alpha}_{N-2-j}\right)\left[\cos\left(T-t_{j+1}\right)-\cos\left(T-t_j\right)\right]\\
		&=\sum_{k=1}^{4}\mathcal{S}_k,
	\end{align*}
where
	\begin{align*}
		\mathcal{S}_1&:=hb_1^2\sum_{j=0}^{N-1}\widehat{\alpha}^2_{N-1-j}+h\gamma^2\sum_{j=0}^{N-1}\widehat{\alpha}^2_{N-2-j},\\
		\mathcal{S}_2&:=2hb_1\gamma\sum_{j=0}^{N-1}\widehat{\alpha}_{N-1-j}\widehat{\alpha}_{N-2-j},\\
		\mathcal{S}_3&:=-2b_1\sum_{j=0}^{N-1}\widehat{\alpha}_{N-1-j}\left(\cos\left(T-t_{j+1}\right)-\cos\left(T-t_{j}\right)\right),\\
		\mathcal{S}_4&:=2\gamma\sum_{j=0}^{N-1}\sum_{k=0}^{1}\widehat{\alpha}_{N-2-j}(-1)^{k}\cos\left(T-t_{j+k}\right)+\frac{T}{2}-\frac{\sin(2T)}{4}.
	\end{align*}
For the term $\mathcal{S}_1$, it is clear that
	\begin{align*}
		\mathcal{S}_1
		&=\frac{hb_1^2}{\sin^2\left(\xi\right)}\sum_{k=0}^{N-1}\left(\det\left(A\right)\right)^k\sin^2\left(\left(k+1\right)\xi\right)\\
		&\quad+\frac{h\gamma^2}{\sin^2(\xi)}\sum_{k=0}^{N-2}\left(\det(A)\right)^k\sin^2((k+1)\xi)\\
		&=\frac{hb_1^2}{2\sin^2\left(\xi\right)}\sum_{k=0}^{N-1}\left(\det\left(A\right)\right)^k+\frac{h\gamma^2}{2\sin^2(\xi)}\sum_{k=0}^{N-2}\left(\det(A)\right)^k\\
		&\quad-\frac{hb_1^2}{2\sin^2\left(\xi\right)}\sum_{k=1}^{N}\left(\det\left(A\right)\right)^{k-1}\cos\left(2k\xi\right)\\
		&\quad-\frac{h\gamma^2}{2\sin^2(\xi)}\sum_{k=1}^{N-1}\left(\det(A)\right)^{k-1}\cos(2k\xi).
	\end{align*}

Similarly, we obtain
	\begin{align*}
		\mathcal{S}_2
		&=\frac{hb_1\gamma}{\sin(\xi)\tan(\xi)}\sum_{k=0}^{N-1}\left(\det(A)\right)^{\frac{2k-1}{2}}(1-\cos(2k\xi+2\xi))\\
		&\quad-\frac{hb_1\gamma}{\sin(\xi)\left(\det(A)\right)^{\frac{3}{2}}}\sum_{k=1}^{N}\left(\det(A)\right)^k\sin(2k\xi),\\
		\mathcal{S}_3
		&=-\frac{b_1}{\tan(\xi)}\sum_{k=1}^{N-1}\sum_{i=1}^{2}\left(\det(A)\right)^{\frac{k}{2}}\sin(k\xi+(-1)^ikh)\\
		&\quad-2b_1-b_1\sum_{k=1}^{N-1}\sum_{i=1}^{2}\left(\det(A)\right)^{\frac{k}{2}}\cos(k\xi+(-1)^ikh)\\
		&\quad+\frac{b_1}{\sin(\xi)}\sum_{k=1}^{N}\sum_{i=1}^{2}\left(\det(A)\right)^{\frac{k-1}{2}}\sin(k\xi+(-1)^ikh),\\
		\mathcal{S}_4
		&=\frac{\gamma(\cos(h)-1)}{\sin(\xi)}\sum_{k=1}^{N-1}\sum_{i=1}^{2}\left(\det(A)\right)^{\frac{k-1}{2}}\sin(k\xi+(-1)^ikh)\\
		&\quad+\frac{\gamma\sin(h)}{\sin(\xi)}\sum_{k=1}^{N-1}\sum_{i=1}^{2}(-1)^i\left(\det(A)\right)^{\frac{k-1}{2}}\cos(k\xi+(-1)^ikh)\\
		&\quad+\frac{T}{2}-\frac{\sin(2T)}{4}.
	\end{align*}
Further, we sum up $\mathcal{S}_1$-$\mathcal{S}_4$ to obtain
	\begin{align}\label{eq:vareN}
	\frac{{\rm Var}(e_N)}{\alpha^2}
	&=\frac{\gamma(\cos(h)-1)}{\sin(\xi)}\sum_{k=1}^{N-1}\sum_{i=1}^{2}\left(\det(A)\right)^\frac{k-1}{2}\sin(k\xi+(-1)^ikh)\notag\\
	&\quad-\frac{b_1}{\tan(\xi)}\sum_{k=1}^{N-1}\sum_{i=1}^{2}\left(\det(A)\right)^\frac{k}{2}\sin(k\xi+(-1)^ikh)+\frac{T}{2}\notag\\
	&\quad+\frac{\gamma\sin(h)}{\sin(\xi)}\sum_{k=1}^{N-1}\sum_{i=0}^{1}\left(\det(A)\right)^\frac{k-1}{2}(-1)^i\cos(k\xi+(-1)^ikh)\notag\\
	&\quad-b_1\sum_{k=1}^{N-1}\sum_{i=0}^{1}\left(\det(A)\right)^\frac{k}{2}\cos(k\xi+(-1)^ikh)\notag\\
	&\quad+\frac{b_1}{\sin(\xi)}\sum_{k=1}^{N}\sum_{i=1}^{2}\left(\det(A)\right)^\frac{k-1}{2}\sin(k\xi+(-1)^ikh)\\
	&\quad+\frac{h b_1^2}{2\sin^2\left(\xi\right)}\sum_{k=0}^{N-1}\left(\det(A)\right)^k(1-\cos(2k\xi+2\xi))\notag\\
	&\quad+\frac{hb_1\gamma}{\sin(\xi)\tan(\xi)}\sum_{k=0}^{N-1}\left(\det(A)\right)^{\frac{2k-1}{2}}(1-\cos(2k\xi+2\xi))\notag\\
	&\quad+\frac{h\gamma^2}{2\sin^2(\xi)}\sum_{k=0}^{N-2}\left(\det(A)\right)^k(1-\cos(2k\xi+2\xi))-2b_1\notag\\
	&\quad-\frac{hb_1\gamma}{\sin(\xi)}\sum_{k=1}^{N}\left(\det(A)\right)^{\frac{2k-3}{2}}\sin(2k\xi)-\frac{\sin(2T)}{4}.\notag
\end{align}

To proceed, we need the following fundamental lemma.
\begin{lemm}\label{sum_lemma}
	For arbitrary $\xi\in(0,2\pi),n\in\mathbb{N}^+$, and $a\in\mathbb{R}$, we have the followings.\\
	{\rm (1)}	\begin{align*}
		\sum_{k=1}^{n}\sin(k\xi)a^k=\frac{a\sin(\xi)-a^{n+1}\sin((n+1)\xi)+a^{n+2}\sin(n\xi)}{1-2a\cos(\xi)+a^2}.
	\end{align*}
	In particular, if $a=1$, then
	\begin{align*}
		\sum_{k=1}^{n}\sin(k\xi)=\frac{\cos(\frac{\xi}{2})-\cos((n+\frac{1}{2})\xi)}{2\sin(\frac{\xi}{2})}.
	\end{align*}
	{\rm (2)} 	\begin{align*}
		\sum_{k=1}^{n}\cos\left(k\xi\right)a^k=\frac{a\cos\left(\xi\right)-a^2-a^{n+1}\cos\left(\left(n+1\right)\xi\right)+a^{n+2}\cos\left(n\xi\right)}{1-2a\cos\left(\xi\right)+a^2}.	
	\end{align*}
	In particular, if $a=1$, then
	\begin{align*}
		\sum_{k=1}^{n}\cos(k\xi)=\frac{\sin((n+\frac{1}{2})\xi)}{2\sin(\frac{\xi}{2})}-\frac{1}{2}.
	\end{align*}
\end{lemm}
\begin{proof}
	(1) See \cite[Lemma 3.1]{suijizhenzidapiancha}.\\
	(2) Since $\cos(k\xi)=\big(e^{-\mathbf{i}k\xi}+e^{\mathbf{i}k\xi}\big)/2$, we have
	\begin{align*}
		&\sum_{k=1}^{n}\cos\left(k\xi\right)a^k=\sum_{k=1}^{n}\frac{a^ke^{\mathbf{i}k\xi}+a^ke^{-\mathbf{i}k\xi}}{2}=\frac{1}{2}\sum_{k=1}^{n}\left(ae^{\mathbf{i}\xi}\right)^k+\frac{1}{2}\sum_{k=1}^{n}\left(ae^{-\mathbf{i}\xi}\right)^k\\
		&=\frac{1}{2}\frac{ae^{\mathbf{i}\xi}\left(1-\left(ae^{\mathbf{i}\xi}\right)^n\right)}{1-ae^{\mathbf{i}\xi}}+\frac{1}{2}\frac{ae^{-\mathbf{i}\xi}\left(1-\left(ae^{-\mathbf{i}\xi}\right)^n\right)}{1-ae^{-\mathbf{i}\xi}}\\
		&=\frac{a\left(e^{\mathbf{i}\xi}+e^{-\mathbf{i}\xi}\right)-2a^2-a^{n+1}\left(e^{\mathbf{i}(n+1)\xi}+e^{-\mathbf{i}(n+1)\xi}\right)+a^{n+2}(e^{\mathbf{i}n\xi}+e^{-\mathbf{i}n\xi})}{2\left(1-a\left(e^{-\mathbf{i}\xi}+e^{\mathbf{i}\xi}\right)+a^2\right)}\\
		&=\frac{a\cos(\xi)-a^2-a^{n+1}\cos((n+1)\xi)+a^{n+2}\cos(n\xi)}{1-2a\cos(\xi)+a^2}.
	\end{align*}
	If $a=1$, it yields that
	\begin{align*}
		\sum_{k=1}^{n}\cos(k\xi)&=\frac{\cos(\xi)-1-\cos((n+1)\xi)+\cos(n\xi)}{2-2\cos\left(\xi\right)}\\
		&=\frac{\cos\left(n\xi\right)-\cos\left(\left(n+1\right)\xi\right)}{4\sin^2\left(\frac{\xi}{2}\right)}-\frac{1}{2}=\frac{\sin\left(\left(n+\frac{1}{2}\right)\xi\right)}{2\sin\left(\frac{\xi}{2}\right)}-\frac{1}{2}.
	\end{align*}
	The proof is completed.
\end{proof}

The expression \eqref{eq:vareN} of the variance {\rm Var}($e_N$) depends on $\det(A)$. Recall that \eqref{general_numerical_zz} is a symplectic method if and only if $\det(A)=1$ (see e.g., \cite{review2015}). Hence, we are in the position to simplify the expressions of the variances of the errors for symplectic and non-symplectic methods, separately.
By means of Lemma \ref{sum_lemma},
we simplify ${\rm Var}(e_N)$ further to get the following statements.
\begin{prop}\label{varforsym}
	For the symplectic method with $\xi\neq h$, the variance of the error $e_N$ is given by
	\begin{align}\label{jnm}
	{\rm Var}(e_N)
		&= \frac{\alpha^2b_1^2\left(2T+h\right)}{4\sin^2\left(\xi\right)}
		-\frac{\alpha^2hb_1^2\sin\left(\left(2N+1\right)\xi\right)}{4\sin^3\left(\xi\right)}-\frac{\alpha^2\sin\left(2T\right)}{4}\notag\\
		&\quad+\frac{\alpha^2b_1(Z_3^++Z_3^-)}{\sin\left(\xi\right)}-\frac{\alpha^2h\gamma^2\sin\left(\left(2N-1\right)\xi\right)}{4\sin^3\left(\xi\right)}+\frac{\alpha^2}{2}T\notag\\
		&\quad-\frac{\alpha^2b_1h\gamma\sin\left(2N\xi\right)}{2\sin^3\left(\xi\right)}+\frac{\alpha^2b_1\gamma T}{\sin\left(\xi\right)\tan\left(\xi\right)}+\frac{\alpha^2\gamma^2\left(2T-h\right)}{4\sin^2\left(\xi\right)}\\
		&\quad+\frac{\alpha^2\gamma\sin\left(h\right)}{\sin\left(\xi\right)}(Z_2^{+}-Z_2^{-})-\alpha^2b_1(Z_2^++Z_2^-)\notag\\
		&\quad-\Big(\frac{2\alpha^2\gamma\sin^2\left(\frac{h}{2}\right)}{\sin\left(\xi\right)}+\frac{\alpha^2b_1}{\tan(\xi)}\Big)(Z_1^{+}+Z_1^{-})-\alpha^2b_1,\notag
	\end{align}
	where
	\begin{align*}
		&Z_1^{\pm}=\frac{\cos(\frac{1}{2}\xi\pm \frac{1}{2}h)-\cos((N-\frac{1}{2})(\xi\pm h))}{2\sin(\frac{1}{2}\xi\pm \frac{1}{2}h)},\\ 
		&Z_2^\pm=\frac{\sin((N-\frac{1}{2})(\xi\pm h))}{2\sin(\frac{1}{2}\xi\pm \frac{1}{2}h)},\\
		&Z_3^\pm=\frac{\cos(\frac{1}{2}\xi\pm \frac{1}{2}h)-\cos((N+\frac{1}{2})(\xi\pm h))}{2\sin(\frac{1}{2}\xi\pm \frac{1}{2}h)}.
	\end{align*}
	For the symplectic method with $\xi=h$,
	\begin{equation}\label{hms}
		\begin{split}
		{\rm Var}(e_N)
		&=\frac{\alpha^2b_1^2(2T+h)}{4\sin^2\left(h\right)}+\alpha^2\gamma\Big(\frac{\sin\left(2T-h\right)}{2\sin\left(h\right)}-N+\frac{1}{2}\Big)\\
		&\quad+\frac{\alpha^2\gamma^2(2T-h)}{4\sin^2\left(h\right)}-\frac{\alpha^2h\gamma^2\sin\left(2T-h\right)}{4\sin^3\left(h\right)}-\alpha^2b_1N\\
		&\quad+\frac{\alpha^2b_1\gamma T}{\sin\left(h\right)\tan\left(h\right)}-\frac{\alpha^2b_1h\gamma\sin\left(2T\right)}{2\sin^3\left(h\right)}-\frac{\alpha^2\sin\left(2T\right)}{4}\\
		&\quad-\frac{\alpha^2b_1\left(\cos\left(h\right)-\cos\left(2T-h\right)\right)}{2\tan\left(h\right)\sin\left(h\right)}-\frac{\alpha^2hb_1^2\sin\left(2T+h\right)}{4\sin^3\left(h\right)}\\
		&\quad+\frac{\alpha^2b_1\left(\cos\left(h\right)-\cos\left(2T+h\right)\right)}{2\sin^2\left(h\right)}-\frac{\alpha^2b_1\sin\left(2T-h\right)}{2\sin\left(h\right)}\\
		&\quad-\frac{\alpha^2\gamma\sin^2\left(\frac{h}{2}\right)\left(\cos\left(h\right)-\cos\left(2T-h\right)\right)}{\sin^2\left(h\right)}+\frac{\alpha^2(T-b_1)}{2}.
			\end{split}
	\end{equation}
\end{prop}

\begin{prop}\label{non-symplectic}
	For the non-symplectic method, the variance of the error $e_N$ is given by
	\begin{align*}
		{\rm Var}(e_N)
		&=\alpha^2h\frac{b_1^2(1-\left(\det\left(A\right)\right)^N)+\gamma^2(1-\left(\det\left(A\right)\right)^{N-1})}{\left(2-2\det\left(A\right)\right)\sin^2\left(\xi\right)}\\
		&\quad-\alpha^2\Big(\frac{b_1\sqrt{\det(A)}}{\tan\left(\xi\right)}+\frac{2\gamma\sin^2\left(\frac{h}{2}\right)}{\sin\left(\xi\right)}\Big)(H_1^+ + H_1^-)\notag\\
		&\quad-\alpha^2 b_1\sqrt{\det(A)}(H_2^+ + H_2^-)+\frac{\alpha^2\gamma\sin(h)}{\sin(\xi)}(H_2^{+}-H_2^{-})\notag\\
		&\quad+\frac{\alpha^2hb_1\gamma(1-(\det(A))^N)}{(1-\det(A))\sin(\xi)\tan(\xi)\sqrt{\det(A)}}-\frac{\alpha^2\sin(2T)}{4}\notag\\
		&\quad-\frac{\alpha^2h\gamma^2H_5}{2\sin^2\left(\xi\right)}-\frac{\alpha^2hb_1\gamma H_6}{\sin(\xi)\sqrt{\det(A)}}+\frac{\alpha^2b_1}{\sin(\xi)}(H_3^{+}+H_3^{-})\notag\\
		&\quad+\frac{\alpha^2T}{2}-2\alpha^2b_1-\Big(\frac{\alpha^2hb_1^2}{2\sin^2(\xi)}+\frac{\alpha^2hb_1\gamma\cos(\xi)}{\sin^2(\xi)\sqrt{\det(A)}}\Big)H_4,\notag
	\end{align*}
	where
	\begin{align*}
		H_{1}^{\pm}&=\frac{\sin\left(\xi\pm h\right)-\left(\det\left(A\right)\right)^{\frac{N-1}{2}}\sin\left(N\xi\pm T\right)}{1-2\sqrt{\det\left(A\right)}\cos\left(\xi\pm h\right)+\det\left(A\right)}\\
		&\qquad+\frac{\left(\det\left(A\right)\right)^{\frac{N}{2}}\sin\left(\left(N-1\right)\xi\pm T\mp h\right)}{1-2\sqrt{\det\left(A\right)}\cos\left(\xi\pm h\right)+\det\left(A\right)},\\
		H_2^{\pm}&=\frac{\cos(\xi\pm h)-\sqrt{\det(A)}-(\det(A))^{\frac{N-1}{2}}\cos(N\xi\pm T)}{1-2\sqrt{\det(A)}\cos(\xi\pm h)+\det(A)}\\
		&\qquad+\frac{(\det(A))^{\frac{N}{2}}\cos((N-1)\xi\pm T\mp h)}{1-2\sqrt{\det(A)}\cos(\xi\pm h)+\det(A)},\\
		H_3^{\pm}&=\frac{\sin(\xi\pm h)+(\det(A))^{\frac{N+1}{2}}\sin(N\xi\pm T)}{1-2\sqrt{\det(A)}\cos(\xi\pm h)+\det(A)}\\
		&\qquad-\frac{(\det(A))^{\frac{N}{2}}\sin((N+1)\xi\pm T\pm h)}{1-2\sqrt{\det(A)}\cos(\xi\pm h)+\det(A)},\\
		H_4&=\frac{\cos(2\xi)-(\det(A))^N\cos(2N\xi+2\xi)}{1-2\det(A)\cos(2\xi)+(\det(A))^2}\\
		&\qquad+\frac{(\det(A))^{N+1}\cos(2N\xi)-\det(A)}{1-2\det(A)\cos(2\xi)+(\det(A))^2},\\
		H_5&=\frac{\cos(2\xi)-\left(\det\left(A\right)\right)^{N-1}\cos\left(2N\xi\right)}{1-2\det\left(A\right)\cos\left(2\xi\right)+\left(\det\left(A\right)\right)^2}\\
		&\qquad-\frac{\det(A)-\left(\det\left(A\right)\right)^N\cos\left(2N\xi-2\xi\right)}{1-2\det\left(A\right)\cos\left(2\xi\right)+\left(\det\left(A\right)\right)^2},\\
		H_6&=\frac{\sin(2\xi)-(\det(A))^N\sin(2N\xi+2\xi)+(\det(A))^{N+1}\sin(2N\xi)}{1-2\det(A)\cos(2\xi)+(\det(A))^2}.
	\end{align*}
\end{prop}

\section{Comparison of symplectic and non-symplectic methods for test equation}\label{chap4}
In subsections \ref{3.1} and \ref{3.2}, we study the exact variance of the asymptotic error distributions for several concrete symplectic and non-symplectic methods for the test equation \eqref{model_zz}.
The superiority of symplectic methods will be discussed in subsection \ref{Sec:comparison}.

\subsection{Error evolution for symplectic methods}\label{3.1}
In this subsection, we focus on the asymptotic error distribution of the errors $\{e_N\}_{N\in\mathbb{N}_+}$ for several symplectic methods, including stochastic $\beta$ methods (see, e.g., \cite[equation (2.7)]{Milstein2004}) and general symplectic methods with $\xi=h$.

\subsubsection{Symplectic $\beta$ methods \,{\rm (}$\beta\in[0,1]${\rm )}}
Applying the symplectic $\beta$ method to \eqref{model_zz}, the coefficients $A_\beta$ and $b_\beta$ are given by
\begin{align*}
	A_{\beta}&=\frac{1}{1+\beta(1-\beta)h^2}\begin{pmatrix}
		1-(1-\beta)^2h^2&h\\-h&1-\beta^2h^2
	\end{pmatrix},\\
	b_\beta&=\frac{1}{1+\beta(1-\beta)h^2}\begin{pmatrix}
		(1-\beta)h\\1
	\end{pmatrix}.
\end{align*}
The symplectic $\beta$ method reduces to the midpoint method when $\beta=\frac{1}{2}$. Notice that the components of $A_\beta$ and $b_\beta$ satisfy \eqref{mzn3}.

It is clear that
\begin{align*}
	\det(A_\beta)&=1,\quad {\rm tr}(A_\beta)=\frac{2-\left(2\beta^2-2\beta+1\right)h^2}{1+\beta\left(1-\beta\right)h^2},\quad 	\gamma=\frac{\beta h}{1+\beta(1-\beta)h^2},\\
	&\sin(\xi)=\frac{h\sqrt{4-(1-2\beta)^2h^2}}{2+2\beta(1-\beta)h^2},\qquad\cos(\xi)=\frac{2-(2\beta^2-2\beta+1)h^2}{2+2\beta(1-\beta)h^2}.
\end{align*}
By substituting the above expressions into \eqref{jnm}, we obtain
	\begin{equation}\label{hdj}
		\begin{split}
		&\quad\frac{{\rm Var}(e_N)}{\alpha^2}\\
		=&\frac{2(1-\beta)^2}{4-(1-2\beta)^2h^2}\Big\{\frac{2T+h}{2}-\frac{(1+\beta(1-\beta)h^2)\sin(2N\xi+\xi)}{(4-(1-2\beta)^2h^2)^{\frac{1}{2}}}\Big\}\\
		&-\frac{2\beta^2(1+(1-\beta)\beta h^2)\sin((2N-1)\xi)}{(4-(1-2\beta)^2h^2)^{\frac{3}{2}}}+\frac{2\alpha^2\beta^2(T-\frac{1}{2}h)}{4-(2\beta-1)^2h^2}\\
		&-\frac{4(1-\beta)\beta(1+\beta(1-\beta)h^2)\sin(2N\xi)}{(4-(1-2\beta)^2h^2)^{\frac{3}{2}}}-\frac{(1-\beta)h}{1+\beta(1-\beta)h^2}\\
		&-\Big(\frac{(1-\beta)(2-(2\beta^2-2\beta+1)h^2)}{(1+\beta(1-\beta)h^2)}+4\beta\sin^2(\frac{h}{2})\Big)\frac{Z_1^++Z_1^-}{\sqrt{4-(1-2\beta)^2h^2}}\\
		&-\frac{(1-\beta)h(Z_2^{+}+Z_2^{-})}{2(1+\beta(1-\beta)h^2)}+\frac{2\beta\sin(h)(Z_2^+-Z_2^-)}{\sqrt{4-(1-2\beta)^2h^2}}+\frac{2T-\sin(2T)}{4}\\
		&+\frac{2(1-\beta)(Z_3^{+}+Z_3^{-})}{\sqrt{4-(1-2\beta)^2h^2}}+\frac{2(1-\beta)\beta T(2-(2\beta^2-2\beta+1)h^2)}{(1+\beta(1-\beta)h^2)(4-(1-2\beta)^2h^2)}.
	\end{split}
	\end{equation}

\textbf{Case 1}. $\beta\in[0,\frac{1}{2}-\frac{\sqrt{6}}{6})\cup(\frac{1}{2}-\frac{\sqrt{6}}{6},\frac{1}{2}+\frac{\sqrt{6}}{6})\cup(\frac{1}{2}+\frac{\sqrt{6}}{6},1]$.

By the Taylor expansion $\arcsin(h)=h+\frac{1}{6}h^3+\frac{3\, }{40}h^5+\mathcal{O}(h^7)$, we have
$\xi=h+ (\frac{{\mathrm{\beta}}^2}{2} - \frac{\mathrm{\beta}}{2} + \frac{1}{24})h^3+ (\frac{3 {\mathrm{\beta}}^4}{8} - \frac{3 {\mathrm{\beta}}^3}{4} + \frac{7 {\mathrm{\beta}}^2}{16} - \frac{\mathrm{\beta}}{16} + \frac{3}{640}) h^5+\mathcal{O}(h^7),$
which implies that
	$
		\sin(\frac{\xi-h}{2})=(\frac{{\mathrm{\beta}}^2}{4} - \frac{\mathrm{\beta}}{4} + \frac{1}{48})\, h^3+(\frac{3\, {\mathrm{\beta}}^4}{16} - \frac{3\, {\mathrm{\beta}}^3}{8} + \frac{7\, {\mathrm{\beta}}^2}{32} - \frac{\mathrm{\beta}}{32} + \frac{3}{1280})\, h^5+\mathcal{O}(h^7).$
As other terms in \eqref{hdj} can be similarly expanded, one has
\begin{align*}
	{\rm Var}(e_N)=(3 {\mathrm{\beta}}^2 - 3\mathrm{\beta} + 1)\Big(\frac{\alpha^2}{6}T+\frac{\alpha^2}{12}\sin(2T)\Big)h^2+\mathcal{O}(h^3).
\end{align*}
Hence we derive
the central limit theorem of the error of the stochastic $\beta$ method
\begin{align*}
	Ne_N-\mathbb{E}[Ne_N]\stackrel{d}{\Rightarrow} \mathcal{N}\Big(0,\alpha^2T^2(3 {\mathrm{\beta}}^2 - 3\mathrm{\beta} +1)\Big(\frac{T}{6}+\frac{\sin(2T)}{12}\Big)\Big).
\end{align*}

\textbf{Case 2}. $\beta\in\{\frac{1}{2}-\frac{\sqrt{6}}{6},\frac{1}{2}+\frac{\sqrt{6}}{6}\}$.

Since $\frac{\beta^2}{4}-\frac{\beta}{4}+\frac{1}{48}=0$, it yields that
$\sin(\frac{\xi-h}{2})=\frac{h^5}{960}+\frac{h^7}{24192}+\mathcal{O}(h^9)$.
By expanding  $\cos(\frac{\xi-h}{2}),\cos((N-\frac{1}{2})(\xi-h)),\sin((N-\frac{1}{2})(\xi-h))$ and $\cos((N+\frac{1}{2})(\xi-h))$ to $\mathcal{O}(h^7)$, we can similarly use Proposition \ref{varforsym} to obtain 
\begin{align*}
	{\rm Var}(e_N)&=\alpha^2\left(\frac{7T}{36}+\frac{\sin(2T)}{16}\right)h^2+\mathcal{O}(h^3).
\end{align*}
Consequently, the following central limit theorem of the error holds 
$$Ne_N-\mathbb{E}[Ne_N]\stackrel{d}{\Rightarrow} \mathcal{N}\Big(0,\alpha^2T^2\Big(\frac{7T}{36}+\frac{\sin(2T)}{16}\Big)\Big).$$

\begin{rema}
	For any fixed $T>0$, the error constant of the midpoint method ($\beta=\frac{1}{2}$) is minimal among symplectic $\beta$ methods. 
\end{rema}

\subsubsection{General symplectic methods with $\xi=h$}
By \eqref{jiaodu}, for the symplectic method, the condition $\xi=h$ is equivalent to ${\rm tr}(A)=2\cos(h)$. By assuming further that the coefficients $a_{ij}(\cdot),b_i(\cdot),i,j=1,2$, are smooth functions satisfying \eqref{mzn3}, we have
\begin{equation}\label{sjx}
	\begin{split}
		a_{11} &= 1+a_{11}^{(1)}h^2+a_{11}^{(2)}h^3+a_{11}^{(3)}h^4+\mathcal{O}(h^5),\\
		a_{22}&=1-(1+a_{11}^{(1)})h^2-a_{11}^{(2)}h^3+(\frac{1}{12}-a_{11}^{(3)})h^4+\mathcal{O}(h^5),\\
		a_{12}&=h+a_{12}^{(1)}h^2+a_{12}^{(2)}h^3+a_{12}^{(3)}h^4+\mathcal{O}(h^5),\\
		a_{21}&=-h+a_{21}^{(1)}h^2+a_{21}^{(2)}h^3+a_{21}^{(3)}h^4+\mathcal{O}(h^5),\\
		b_1&=b_1^{(1)}h+b_1^{(2)}h^2+b_1^{(3)}h^3+b_1^{(4)}h^4+\mathcal{O}(h^5),\\
		b_2&=1+b_2^{(1)}h+b_2^{(2)}h^2+b_2^{(3)}h^3+b_2^{(4)}h^4+\mathcal{O}(h^5).
	\end{split}
\end{equation}
In view of \eqref{hms} and \eqref{sjx}, it holds that
${\rm Var}(e_N)=K^e_Th^2+\mathcal{O}(h^3)$,
where the error constant $K^e_T$ is given by
	\begin{equation}\label{lkj}
		\begin{split}
		K^e_T&=\alpha^2\frac{1+3b_1^{(1)}(b_1^{(1)}-1)+3(a_{12}^{(1)}+b_2^{(1)})^2}{6}T\\
		&\quad+\alpha^2\frac{(1-2b_1^{(1)})(a_{12}^{(1)}+b_2^{(1)})(\cos(2T)-1)}{2}\\
		&\quad+\alpha^2\frac{1+3b_1^{(1)}(b_1^{(1)}-1)-3(a_{12}^{(1)}+b_2^{(1)})^2}{12}\sin\left(2T\right).
				\end{split}
	\end{equation}
Furthermore, the corresponding central limit theorem of the error is 
\begin{align*}
Ne_N-\mathbb{E}[Ne_N]\stackrel{d}{\Rightarrow} \mathcal{N}\left(0,T^2K_T^e\right).
\end{align*}

We present three existing symplectic methods satisfying $\xi=h$ in Table \ref{table:1}.
\begin{table}[htb] 
	\begin{center} 
		\renewcommand{\arraystretch}{1} 
		\resizebox{1\textwidth}{0.085\textheight} { 
			\begin{tabular}{|c|c|c|c|}  
				\hline  \textbf{Symplectic method} & $A$ & $b$ & $K^e_T$ \\  
				\hline  Exponential method & 	
				$\begin{pmatrix}
					\cos(h)&\sin(h)\\-\sin(h)&\cos(h)
				\end{pmatrix}$
				& 
				$\begin{pmatrix}
					0\\1
				\end{pmatrix} $&
				$\frac{\alpha^2}{6}T+\frac{\alpha^2\sin(2T)}{12}$
			\\ 
				\hline  Integral method & $\begin{pmatrix}
					\cos(h)&\sin(h)\\-\sin(h)&\cos(h)
				\end{pmatrix}$ & $\begin{pmatrix}
					\sin(h)\\\cos(h)
				\end{pmatrix} $&
				$\frac{\alpha^2}{6}T+\frac{\alpha^2\sin(2T)}{12}$\\ 
				\hline  Optimal method & $\begin{pmatrix}
					\cos(h)&\sin(h)\\-\sin(h)&\cos(h)
				\end{pmatrix}$ & $\frac{1}{h}\begin{pmatrix}
					2\sin^2(\frac{h}{2})\\\sin(h)
				\end{pmatrix} $ &
				$\frac{\alpha^2}{24}T+\frac{\alpha^2\sin(2T)}{48}$
	 \\
				\hline  
		\end{tabular}}
		\caption{Symplectic methods with $\xi=h$ for \eqref{model_zz}.} \label{table:1} 
	\end{center}  
\end{table}			

\begin{rema}
	When $T\gg1$, among the symplectic methods with $\xi=h$, the numerical method satisfying \begin{align}\label{hlq}
		b_1^{(1)}=\frac{1}{2}\quad\text{and}\quad a_{12}^{(1)}+b_2^{(1)}=0
	\end{align}
	has the minimal error constant. Obviously, the optimal method fulfills \eqref{hlq}. Besides, we construct a symplectic method satisfying \eqref{hlq}:
	\begin{align}\label{mdb}
		A=\begin{pmatrix}
			\cos(h)&\sin(h)\\-\sin(h)&\cos(h)
		\end{pmatrix},\qquad b=\begin{pmatrix}
			\frac{h}{2}\\1
		\end{pmatrix}.
	\end{align}
\end{rema}
\subsection{Error evolution for non-symplectic methods}\label{3.2}
This subsection is devoted to studying the asymptotic error distribution of errors $\{e_N\}_{N\in\mathbb{N}_+}$ for several non-symplectic methods, including stochastic $\theta$-methods $(\theta\neq\frac{1}{2})$ and the PC(EM-BEM) method (see e.g., \cite{review2015}).

\subsubsection{Stochastic $\theta$-methods {\rm (}$\theta\in[0,\frac{1}{2})\cup(\frac{1}{2},1]${\rm )}}
For $\theta\in[0,\frac{1}{2})\cup(\frac{1}{2},1]$, the coefficients of the stochastic $\theta$-method for \eqref{model_zz} are given by
\begin{align*}
	A^\theta=\frac{1}{1+\theta^2h^2}\begin{pmatrix}
		1-(1-\theta)\theta h^2&h\\-h&1-(1-\theta)\theta h^2
	\end{pmatrix},\qquad b^\theta=\frac{1}{1+\theta^2h^2}\begin{pmatrix}
		\theta h\\1
	\end{pmatrix},
\end{align*}
for which \eqref{mzn3} is satisfied. It is readily to show that
\begin{align*}
	\cos(\xi)&=\frac{1-(1-\theta)\theta h^2}{\sqrt{1+(1-\theta)^2h^2}\sqrt{1+\theta^2h^2}},\\
	\sin(\xi)&=\frac{h}{\sqrt{1+(1-\theta)^2h^2}\sqrt{1+\theta^2h^2}},\\
	\det(A^\theta)&=\frac{1+(1-\theta)^2h^2}{1+\theta^2h^2},\quad
	\gamma=\frac{(1-\theta)h}{1+\theta^2h^2}.
\end{align*}
On this basis, we can further formulate $\xi$ in terms of $h$ as follows
$\xi=h+(\theta-\theta^2 - \frac{1}{3})h^3+(\theta^4-2{\mathrm{\theta}}^3+2\theta^2-\theta+ \frac{1}{5})h^5+(3\theta^5-\theta^6-5\theta^4+5{\mathrm{\theta}}^3 - 3\theta^2+\theta-\frac{1}{7})h^7+\mathcal{O}(h^9).$
Plugging the above relations into Proposition \ref{non-symplectic}, 
we can obtain 
\begin{align}\label{mak}
	{\rm Var}(e_N)=K^{\theta}_Th^2+\mathcal{O}(h^3),
\end{align}
where the error constant $K^{\theta}_T$ is given by
\begin{align*}
	K^{\theta}_T
	&=\frac{\alpha^2(2\theta-1)^2}{24}T^3-\frac{\alpha^2\sin(2T)(2\theta-1)^2}{16}T^2\\
	&\quad+\alpha^2\left(\frac{(1-\theta)^3-5\theta^3}{6}+\frac{(2\theta-1)^2\cos(2T)}{16}\right)T\\
	&\quad+\alpha^2\left(\frac{1}{48}+\frac{\left(2\theta-1\right)^2}{32}\right)\sin(2T).
\end{align*}
Thus the error $\{e_N\}_{N\in\mathbb{N}_+}$ of the stochastic $\theta$-method with $\theta\in[0,\frac{1}{2})\cup(\frac{1}{2},1]$ satisfies
the following central limit theorem 
 $$Ne_N-\mathbb{E}[Ne_N]\stackrel{d}{\Rightarrow} \mathcal{N}(0,T^2K_T^\theta).$$
\subsubsection{PC(EM-BEM) method}
The PC(EM-BEM) method is a predictor-corrector method using the Euler--Maruyama method as the predictor and the backward Euler--Maruyama method as the corrector, whose coefficients are given by
\begin{align*}
	A=\begin{pmatrix}
		1-h^2&h\\-h&1-h^2
	\end{pmatrix},\qquad b=\begin{pmatrix}
		h\\1
	\end{pmatrix}.
\end{align*}
By a straightforward calculation, we have
\begin{align*}
	\det(A)&=1-h^2+h^4,\quad {\rm tr}(A)=2-2h^2,\quad \gamma=h^3,\\
	\sin(\xi)&=\frac{h}{\sqrt{1-h^2+h^4}},\quad \cos(\xi)=\frac{1-h^2}{\sqrt{1-h^2+h^4}},
\end{align*}
which leads to $\xi=h+ \frac{2 h^3}{3}+\frac{h^5}{5}+\frac{h^7}{7}+\mathcal{O}(h^9).$ Further, we can simplify the Taylor series in Proposition \ref{non-symplectic} into
	\begin{equation}\label{mzn}
		{\rm Var}(e_N)=\alpha^2\Big(\frac{T^3}{24} -\frac{T^2\sin(2T)}{16} + \frac{6\cos^2(T)+5}{48} T + \frac{5\sin(2T)}{96}\Big)h^2+\mathcal{O}(h^3).
	\end{equation}
Then 
for the PC(EM-BEM) method, 
$Ne_N-\mathbb{E}[Ne_N]$ converges in distribution to 
	$$\mathcal{N}\Big(0,\alpha^2T^2\Big(\frac{T^3}{24} -\frac{T^2\sin(2T)}{16} + \frac{6\cos^2(T)+5}{48} T + \frac{5\sin(2T)}{96}\Big)\Big).$$

\subsection{Superiority of symplectic methods} \label{Sec:comparison}
Let $e^{(s)}_N$ and $e^{(ns)}_N$ be the errors of the considered symplectic method and non-symplectic method for \eqref{model_zz}, respectively.
It has been shown in subsections \ref{3.1} and \ref{3.2} that the error constant $K_T^{(s)}:=K_T\sim T$ for symplectic methods and that  the error constant  $K_T^{(ns)}:=K_T\sim T^3$ for non-symplectic methods, respectively. Since $K_T^{(s)}/K_T^{(ns)}\to 0$ as $T\to\infty$, one can choose a sufficiently large constant $T_0>0$ so that $K_T^{(s)}<K_T^{(ns)}$ for all $T\ge T_0$. 
 
Fix $\epsilon>0$ and $T\ge T_0$. It follows from $K_T^{(s)}<K_T^{(ns)}$ that
		\begin{align*}
				&\lim_{N\rightarrow \infty}\mathbb{P}\big(|Ne^{(s)}_N-\mathbb{E}[Ne^{(s)}_N]|>\epsilon\big)=
				\mathcal{N}(0,T^2K^{(s)}_T)(\{|x|>\epsilon\})\\
				&=\mathcal{N}(0,T^2K^{(ns)}_T)\big(\big\{|x|>\epsilon\sqrt{K^{(ns)}_T/K^{(s)}_T}\big\}\big)
				<\mathcal{N}(0,T^2K^{(ns)}_T)\big(\big\{|x|>\epsilon\big\}\big)\\
				&=\lim_{N\rightarrow \infty}\mathbb{P}\big(|Ne^{(ns)}_N-\mathbb{E}[Ne^{(ns)}_N]|>\epsilon\big).
			\end{align*}
		Therefore there exists $N_1\in\mathbb{N}_+$ such that for any $N\ge N_1$,
\begin{align}\label{eq:Cor1}
		\mathbb{P}(|Ne^{(s)}_{N}-\mathbb{E}[Ne^{(s)}_N]|>\epsilon)<\mathbb{P}(|Ne^{(ns)}_N-\mathbb{E}[Ne^{(ns)}_N]|>\epsilon),\quad\forall N\geq N_0,
		\end{align}

Since $e^{(s)}_N$ and $e^{(ns)}_N$ are Gaussian random variables,
\begin{align*}
	&\quad\lim_{N\rightarrow\infty}\frac{1}{N^2}\log\Big(\frac{\mathbb{P}(|\overline{e}_N|>\epsilon)}{\mathbb{P}(|\hat{e}_N|>\epsilon)}\Big)\\
	&=
	\lim_{N\rightarrow\infty}\frac{1}{N^2}\log\mathbb{P}(|\overline{e}_N|>\epsilon)-\lim_{N\rightarrow\infty}\frac{1}{N^2}\log\mathbb{P}(|\widehat{e}_N|>\epsilon)=-\mathcal{R}_\epsilon(T),
\end{align*}
	where $\mathcal{R}_\epsilon(T):=\frac{\epsilon^2(K_T^{(ns)}-K_T^{(s)})}{2T^2K_T^{(ns)}K_T^{(s)}}>0$.
	Hence there exists some $N_0\in\mathbb{N}_+$ such that
\begin{align}\label{eq:Cor2}
\frac{\mathbb{P}(|\overline{e}_N|>\epsilon)}{\mathbb{P}(|\hat{e}_N|>\epsilon)}\leq \exp\Big(-\frac{1}{2}N^2\mathcal{R}_\epsilon(T)\Big),\qquad\forall N>N_0.
\end{align}

\begin{figure}
	\centering
	\subfigure[The exponential method]{
		\begin{minipage}{0.45\linewidth}
			\centering
			\includegraphics[width=5.7cm]{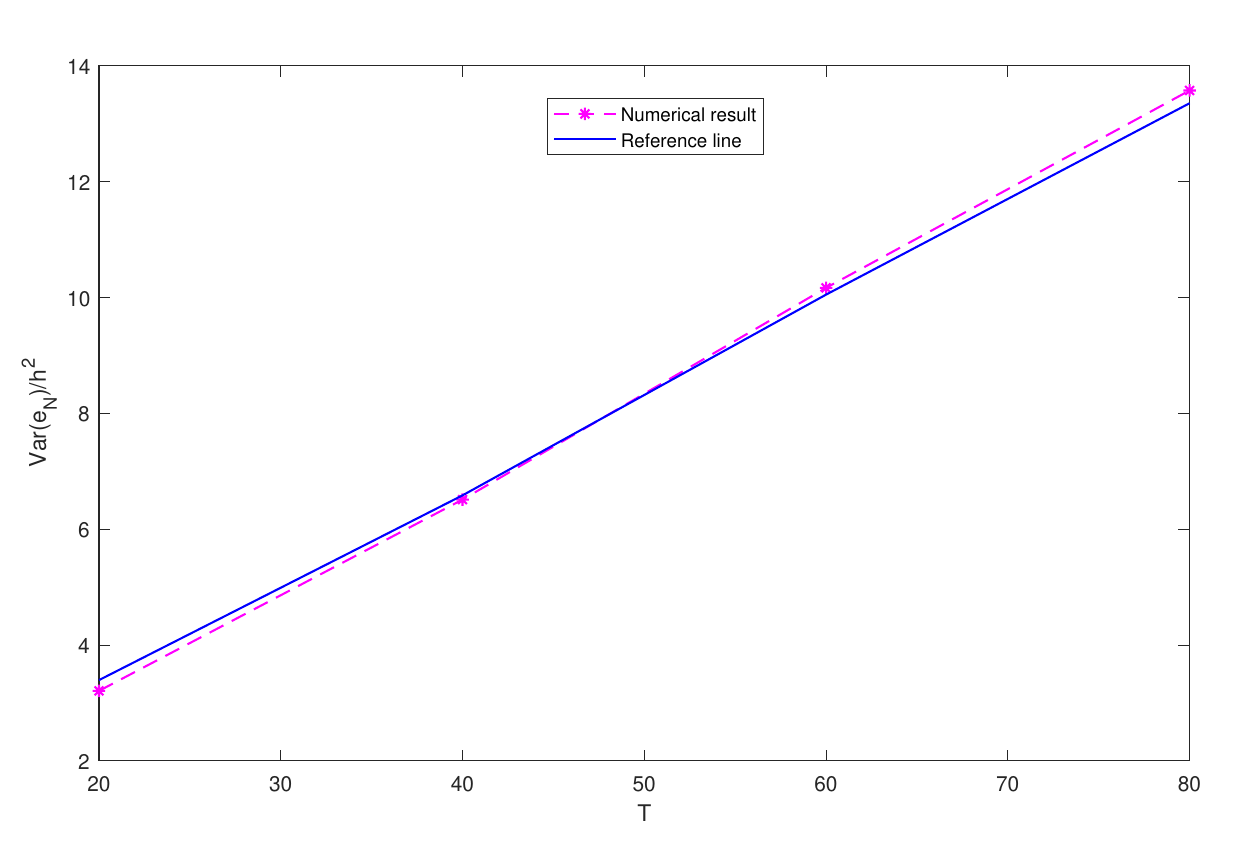}
	\end{minipage}}	
	\subfigure[The integral method]{
		\begin{minipage}{0.45\linewidth}
			\centering
			\includegraphics[width=5.7cm]{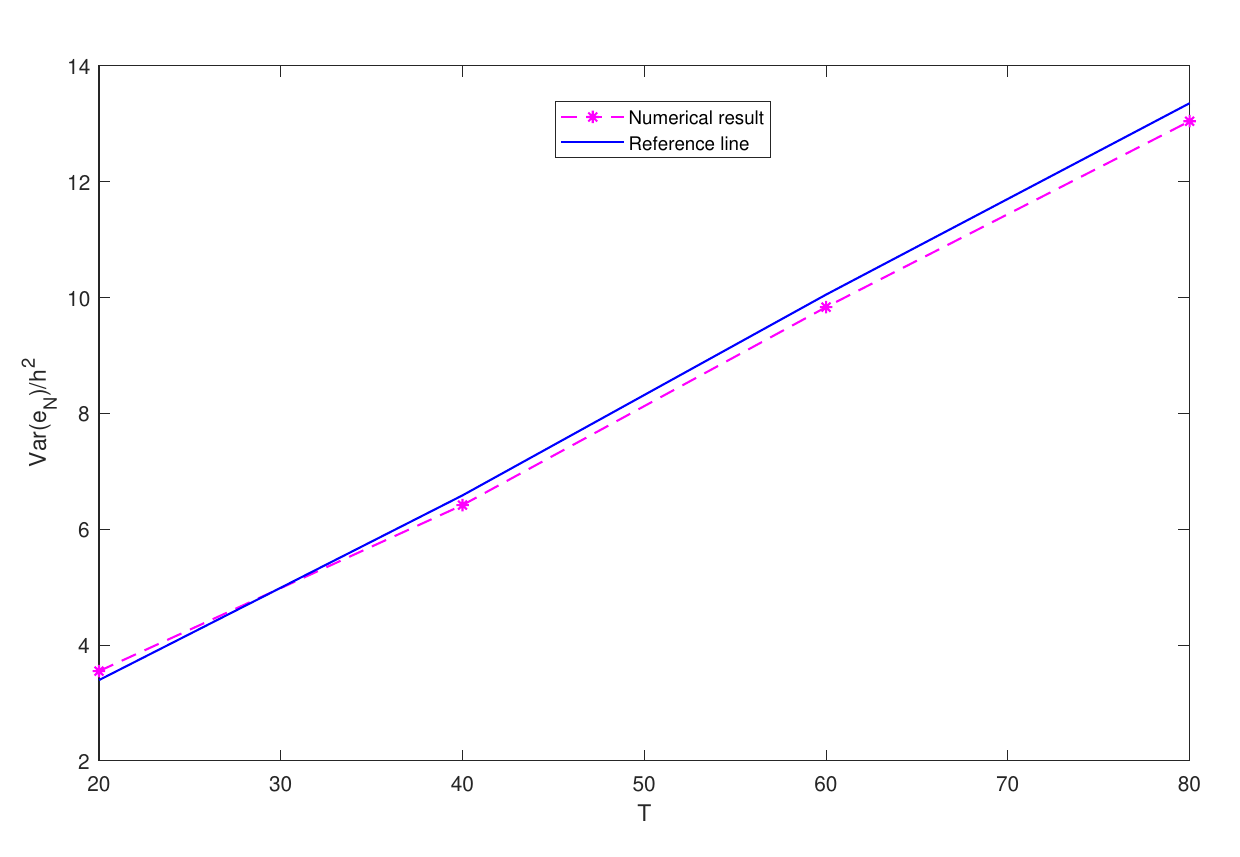}
	\end{minipage}}

	\subfigure[The optimal method]{
		\begin{minipage}{0.45\linewidth}
			\centering
			\includegraphics[width=5.7cm]{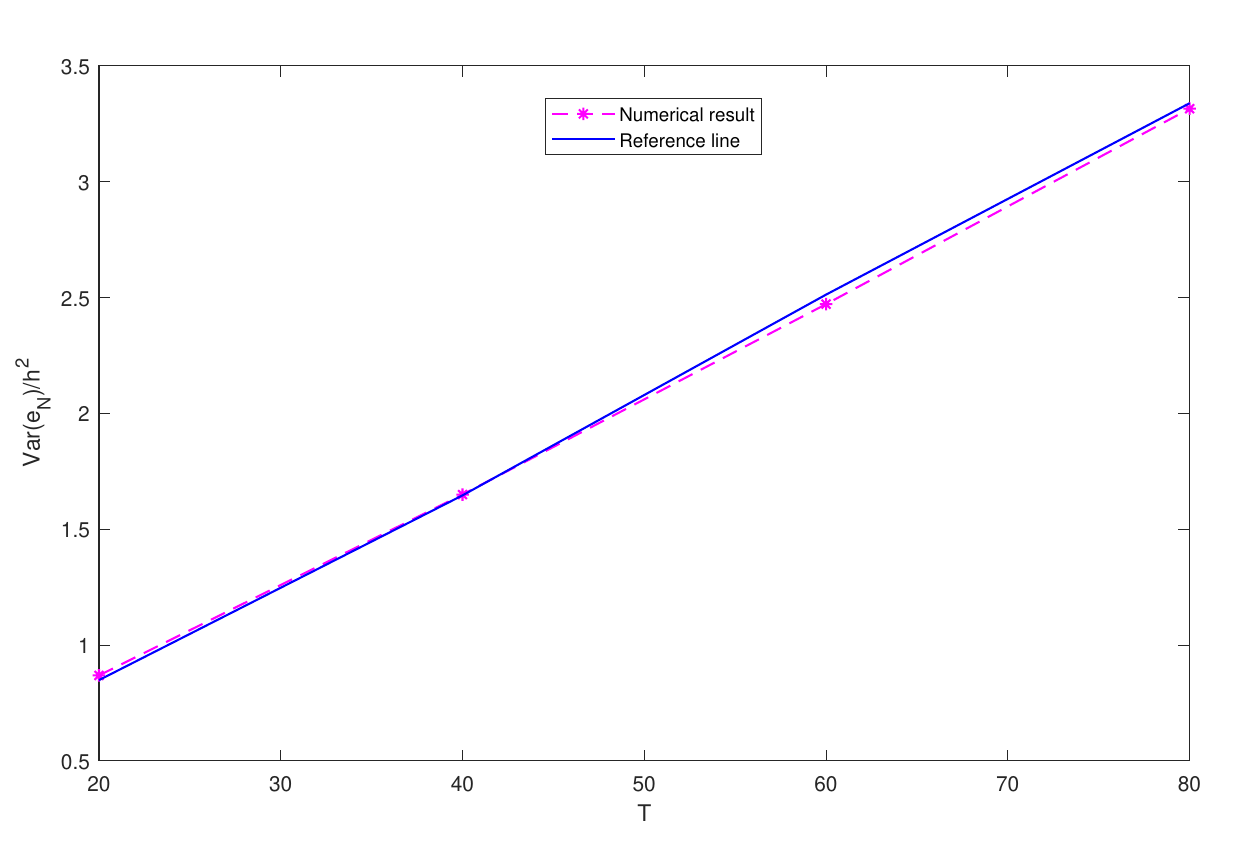}
	\end{minipage}}	
	\subfigure[The method \eqref{mdb}]{
		\begin{minipage}{0.45\linewidth}
			\centering
			\includegraphics[width=5.7cm]{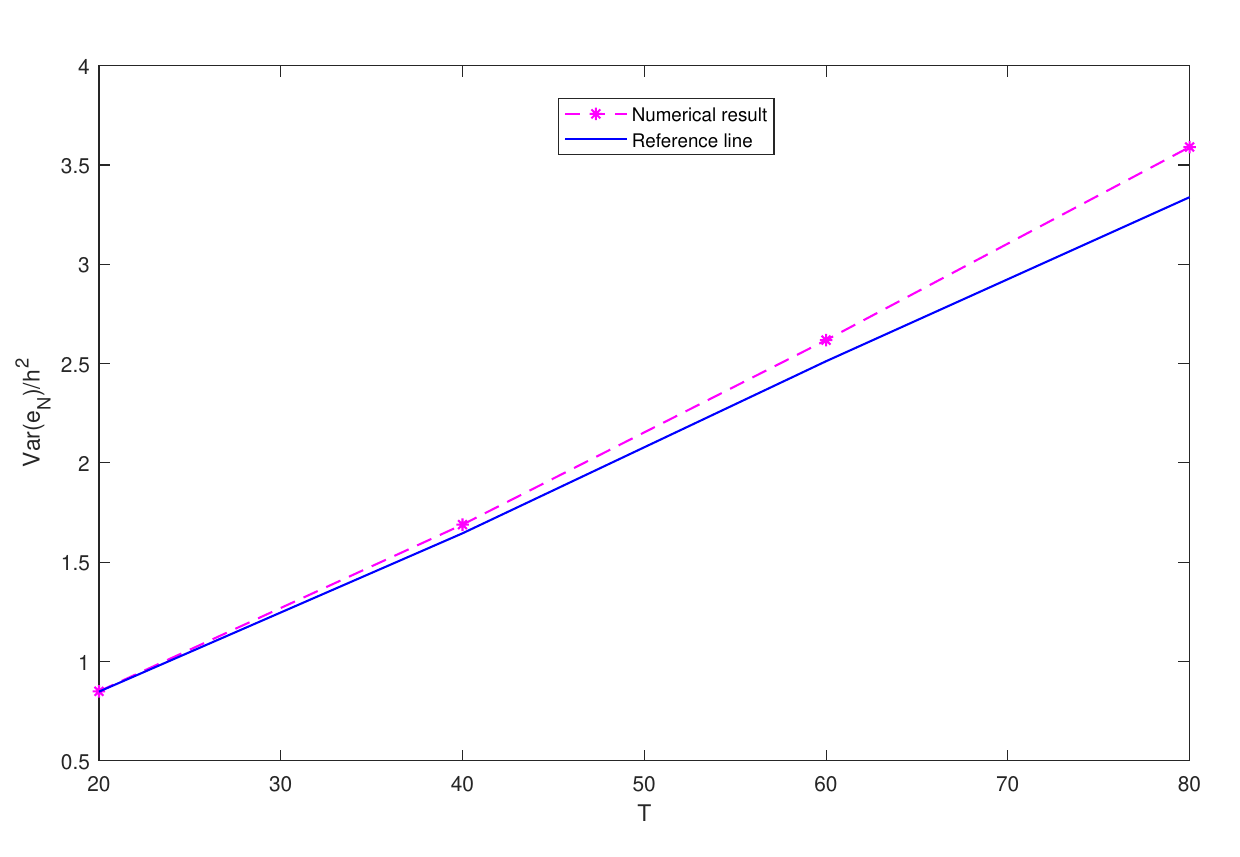}
	\end{minipage}}
	\caption{The relation between Var$(e_N)/h^2$ and $T$ of symplectic methods for the test equation \eqref{model_zz}.}
	\label{km1}
\end{figure}

The above inequality \eqref{eq:Cor1} compares the error's deviation of symplectic and non-symplectic methods for the test equation \eqref{model_zz}. 
The relation \eqref{eq:Cor2} reveals that at the scale $\epsilon$, the probability of the error's deviation from the zero decays exponentially faster for the symplectic method
than that for the non-symplectic method. 
Hence, symplectic methods are superior to non-symplectic methods from the perspective of the asymptotic error distribution, although they may have the same mean square convergence order.

\section{Numerical experiments}\label{chap5}
\begin{figure}
	\centering
	\subfigure[The stochastic $\theta$-method]{
		\begin{minipage}{0.45\linewidth}
			\centering
			\includegraphics[width=5.7cm]{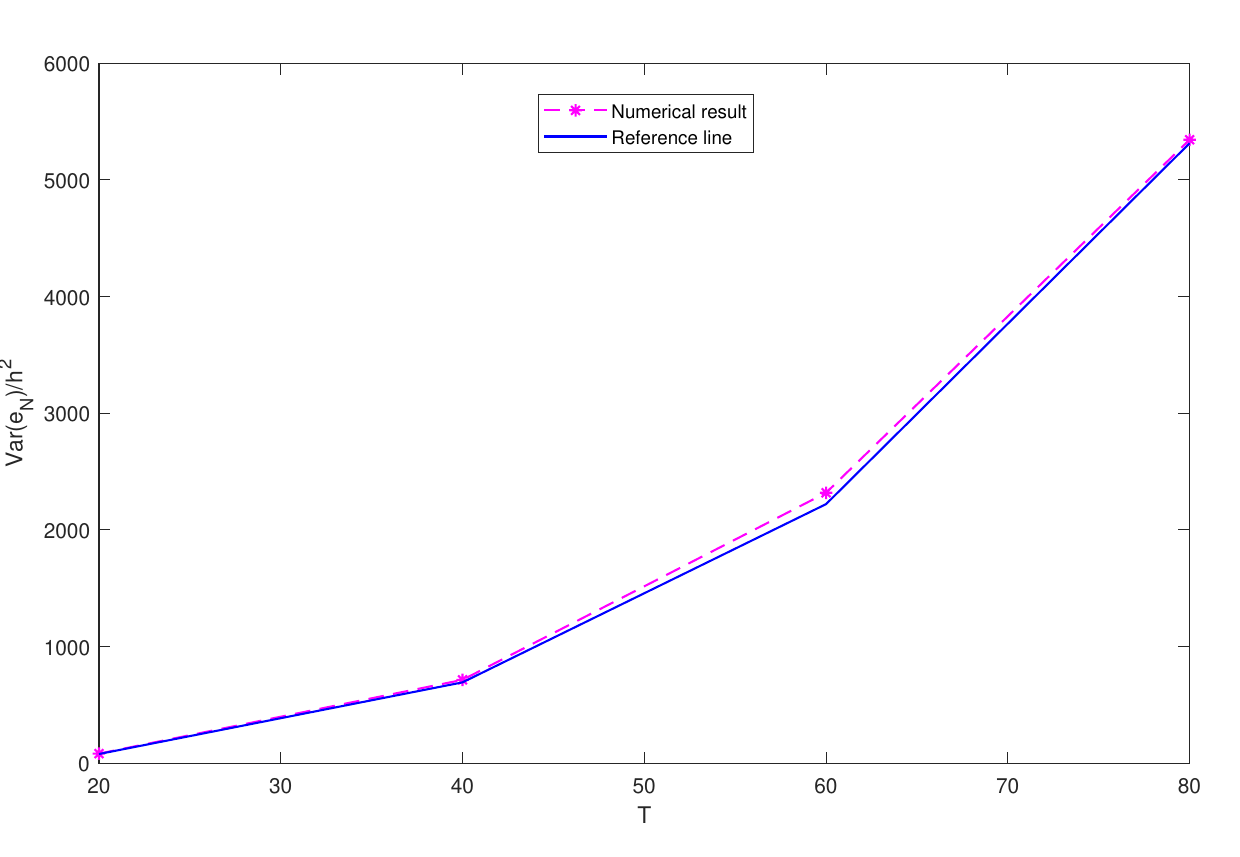}
	\end{minipage}}	
	\subfigure[The PC(EM-BEM) method]{
		\begin{minipage}{0.45\linewidth}
			\centering
			\includegraphics[width=5.7cm]{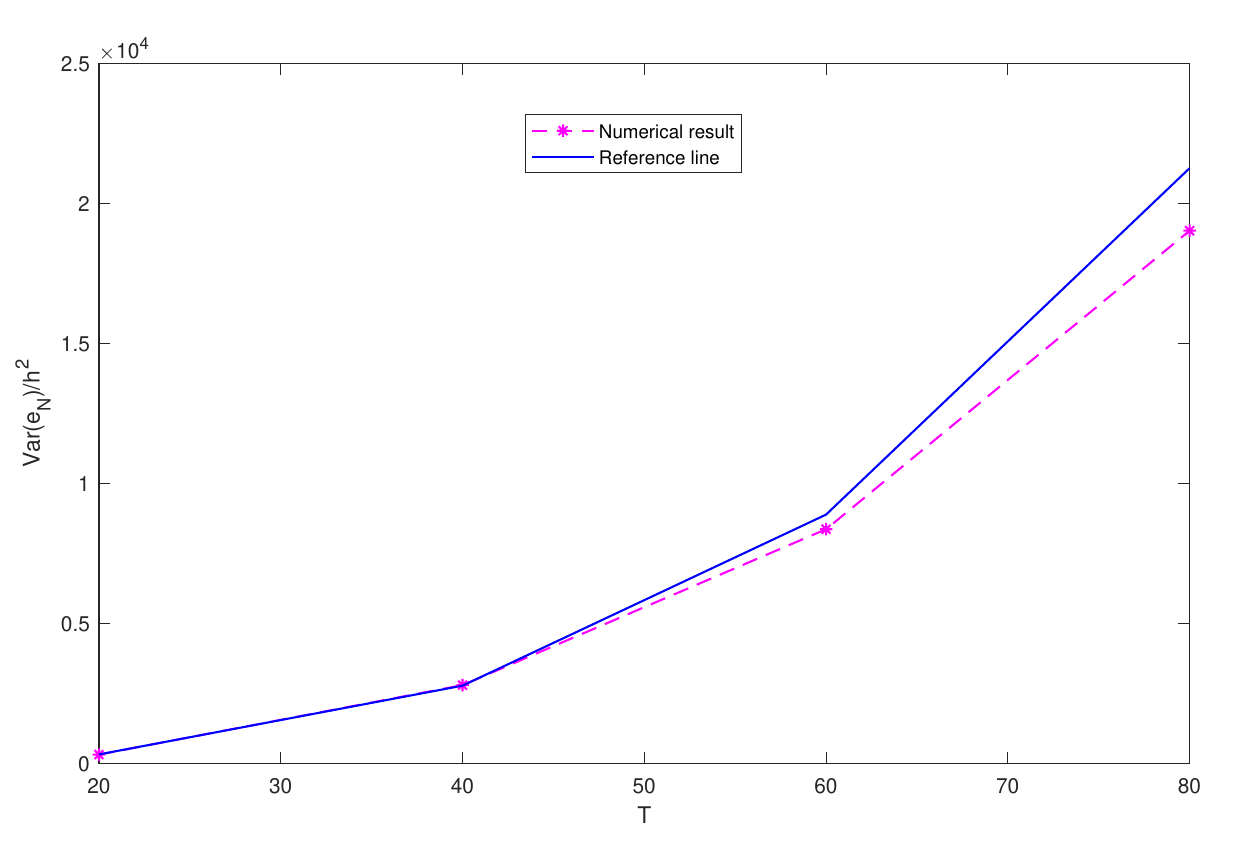}
	\end{minipage}}
	\caption{The relation between Var$(e_N)/h^2$ and $T$ of non-symplectic methods for the test equation \eqref{model_zz}.}
	\label{km2}
\end{figure}
This section provides numerical experiments to illustrate the theoretical results by numerically simulating the exponential method, integral method, optimal method, method \eqref{mdb}, $\theta$ method ($\theta=\frac{1}{4}$) and PC(EM-BEM) method for the test equation \eqref{model_zz}. In the following experiments, we set the initial data $(x_0,y_0)=(1,0)$ and $\alpha=1$. 

First, we show the relation between ${\rm Var}(e_N)/h^2$ and $T$, evaluated with times $20,40,60,80$ by $2000$ sample paths. 
For the exponential method, integral method, optimal method and method \eqref{mdb}, we choose $N=2^7$ with the corresponding time step-sizes $(\frac{5}{2^5},\frac{5}{2^4},\frac{15}{2^5},\frac{5}{2^3})$. As is displayed in Fig. \ref{km1}, 
${\rm Var}(e_N)/h^2$ for these symplectic methods grows linearly with respect to time $T$, which is consistent with the reference line $g(T)=K^e_T\sim T$, where $K^e_T$ is given by \eqref{lkj}. For the stochastic $\theta$-method and PC(EM-BEM) method, we consider the time step-sizes $(\frac{5}{2^{13}},\frac{5}{2^{12}},\frac{15}{2^{13}},\frac{5}{2^{11}})$ by taking $N=2^{15}$. Fig. \ref{km2}
shows that ${\rm Var}(e_N)/h^2\sim T^3$, which coincides with the theoretical results \eqref{mak} and \eqref{mzn}.

\bibliographystyle{plain}
\bibliography{references}
\medskip
Received xxxx 20xx; revised xxxx 20xx; early access xxxx 20xx.
\medskip

\end{document}